\documentclass[10pt]{amsart}
\usepackage{amsmath,amsfonts,latexsym,amssymb,amscd,
  graphicx,amsthm, mathtools}
\usepackage{enumerate}
\usepackage{ulem}
\usepackage{cancel}
\usepackage[usenames]{color}

\usepackage{hyperref}

\allowdisplaybreaks[1]

\newcommand{\Pc}{\mathcal{P}}
\newcommand{\aseq}{{\stackrel{\mathrm{a.s.}}{=}}}
\newcommand{\disteq}{{\stackrel{\mathrm{d}}{=}}}

\newcommand{\ONE}{{\mathbf{1}}}

\newcommand{\N}{{\mathbb N}}

\newcommand{\Fc}{\mathcal{F}}
\newcommand{\Nc}{\mathcal{N}}
\newcommand{\Pp}{\mathsf{P}}
\newcommand{\Z}{{\mathbb Z}}
\newcommand{\iid}{i.i.d.\ }
\newcommand{\RZ}{\R\times\Z}
\newcommand{\ZR}{\Z\times\R}

\newcommand{\E}{\mathsf{E}}
\newcommand{\Leb}{\mathrm{Leb}}

\newcommand{\R}{{\mathbb R}}

\newcommand{\HH}{{\mathbb H}}
\newcommand{\GG}{{\mathbb G}}

\newcommand{\eps}{{\varepsilon}}

\newcommand{\argmin}{\mathop{\mathrm{argmin}}}

 \newcommand{\be}{\kappa^{-1}}
 \newcommand{\temp}{\kappa}
 \newcommand{\visc}{\frac{\kappa}{2}}

\newtheorem{theorem}{Theorem}[section]

\newtheorem{lemma}{Lemma}[section]

\numberwithin{equation}{section}

\renewenvironment{proof}[1][Proof]{
{\noindent {\sc #1: }}
}{
{{\hfill $\Box$ \smallskip}}
}
\makeatletter
\let\orgdescriptionlabel\descriptionlabel
\renewcommand*{\descriptionlabel}[1]{%
  \let\orglabel\label
  \let\label\@gobble
  \phantomsection
  \edef\@currentlabel{#1}%
  \let\label\orglabel
  \orgdescriptionlabel{#1}%
}
\makeatother

\title[Zero-temperature Limit for Directed Polymers and Burgers Equation]{Zero temperature limit for directed polymers and inviscid limit for stationary
solutions of stochastic Burgers Equation}
\author{Yuri Bakhtin}
\address{Courant Institute of Mathematical Sciences\\ New York University \\ 251 Mercer St, New York, NY 10012 }
\author{Liying Li}
\email{bakhtin@cims.nyu.edu, liying@cims.nyu.edu}

\begin{document}
 \maketitle

\begin{abstract}
We consider a space-continuous and time-discrete polymer model for positive temperature and the 
associated zero temperature  model of last passage percolation type.  In our previous work, we constructed and studied infinite-volume polymer measures and one-sided infinite minimizers for the associated variational principle, and used these objects for the study of global stationary solutions of the Burgers
equation with positive or zero viscosity and random kick forcing, on the entire real line.

In this paper, we prove that in the zero temperature limit, the infinite-volume polymer measures 
concentrate on the one-sided minimizers and that the associated global solutions of the 
viscous Burgers equation with random kick forcing converge to the  global solutions of the
inviscid equation.
\end{abstract}

\section{Introduction}

Various models of directed polymers in random environment along with their zero temperature counterparts
of last passage percolation type have been studied actively in recent decades, see, e.g., books
\cite{Hollander:MR2504175}, \cite{Giacomin:MR2380992}, \cite{Comets:MR3444835} and multiple references therein.  On finite
time intervals, positive temperature polymer measures are defined as Gibbs distributions with a random walk as a free measure
and Boltzmann--Gibbs weights given by the potential accumulated by random walk paths from the
random environment. The corresponding zero temperature models are defined in terms of the energy
minimizing paths.

The most interesting questions concern the large time behavior of the random polymer distributions and
energy minimizers. In particular, it is believed that a large family of models of this kind with fast decorrelation of the stationary random potential in dimension $1+1$ belongs to the KPZ universality class, i.e., satisfies limit theorems under scalings with characteristic exponents $2/3$ and $1/3$
and distributional limits of Tracy--Widom type.

Another basic question concerns the infinite-volume Gibbs distributions for polymer measures and the corresponding ground states, i.e., infinite one-sided or two-sided energy minimizers that are usually called geodesics in the literature on last passage percolation (LPP) and first passage percolation (FPP).
There are numerous results concerning infinite geodesics. In particular, existence-uniqueness of one-sided planar geodesics 
with fixed slope and certain geometric features of the joint behavior of different geodesics are known for several models, see \cite{HoNe3}, \cite{HoNe2}, \cite{HoNe},   \cite{Wu}, \cite{CaPi}, \cite{CaPi-ptrf}, \cite{DamronHanson2014}, \cite{kickb:bakhtin2016}, \cite{Rassoul-Agha--Seppalainen--Yilmaz:2013arXiv1311.3016G},
\cite{Rassoul-Agha--Seppalainen--Yilmaz:2014arXiv1410.4474R}, \cite{Rassoul-Agha--Seppalainen--Yilmaz:2015arXiv151000859G}, and a survey \cite{AHD:2015arXiv151103262A}. However, results about thermodynamic limits for directed polymers are relatively new. The first explicit result of this kind known to us is~\cite{Bakhtin-Khanin:MR2791052}, where 
instead of stationarity a localization condition was imposed on the random potential, so the thermodynamic limit is a random measure on paths with random localization radius. More recently, in 
\cite{Georgiou--Rassoul-Agha--Seppalainen--Yilmaz:MR3395462}, \cite{Rassoul-Agha--Seppalainen--Yilmaz:2013arXiv1311.3016G}
 thermodynamic limits were constructed and studied under certain conditions that were
verified for an exactly solvable lattice model called log-gamma polymer, and certain weak disorder models.  

The first complete set of results for a $1+1$-dimensional model that is not exactly solvable were obtained in~\cite{2016arXiv160704864B}, where time-discrete  and space-continuous polymers based on Gaussian random walks were considered. It was shown that for any positive temperature and any fixed asymptotic slope, with probability one, there exists a family of infinite-volume polymer measures satisfying DLR conditions, concentrated on one-sided infinite paths with prescribed asymptotic slope, and indexed by the endpoint. Moreover, it was shown that these infinite-volume Gibbs measures are almost surely uniquely defined and that they are limits of various kinds  (point-to-point, point-to-line, point-to-distribution) of finite-volume polymer measures. It was also shown that 
the total variation distance between
projections
of different polymer measures with the same asymptotic slope 
 on distant coordinates is asymptotically zero, so they 
tend to overlap and can be effectively coupled.
The results crucially depend on the explicitly known form of the dependence of the free energy density,
(also known as the shape function) on the asymptotic slope. Namely, the shape function is quadratic and thus has uniform curvature.

The first main goal of the present paper is to study the zero-temperature asymptotics of the infinite-volume polymer measures constructed in~\cite{2016arXiv160704864B}. In the finite-volume setting, the asymptotic concentration of Gibbs distributions around finite volume ground states, i.e., energy minimizers, is well-known. In the infinite-volume setting, the energies of
paths are infinite, but it is natural to expect that the infinite one-sided minimizers or geodesics (infinite paths whose restrictions on any finite intervals are minimizers) are relevant for this problem. The existence-uniqueness and joint behavior of one-sided minimizers for the same model was studied
in~\cite{kickb:bakhtin2016}. 

In the present paper, we prove that in the zero-temperature limit, with probability
one, the random infinite volume polymer measures converge to delta-measures concentrated on one-sided minimizers. To the best of our knowledge, this is the first result on zero-temperature limit for
infinite directed polymers to appear in literature. In a sense, given the results 
of~\cite{kickb:bakhtin2016} and~\cite{2016arXiv160704864B},
it amounts to interchanging the order of zero-temperature and infinite-volume (or time horizon) limits.

Papers~\cite{kickb:bakhtin2016} and~\cite{2016arXiv160704864B} were, in fact, primarily motivated by the
ergodic program for randomly forced Burgers equation which is a basic nonlinear evolution equation
that has mutiple connections to various problems from traffic modeling to the large scale structure
of the Universe.  It has  interpretations via fluid dynamics and growth models, and we often use the fluid dynamics interpretation where the equation describes the evolution of velocity fields of moving particles.
It is also is tightly related to Hamilton--Jacobi--Bellman (HJB) equations and can be solved with usual HJB methods. 
  
  The viscosity parameter of the Burgers equation can be interpreted as temperature. In fact, if the viscosity is positive, the Burgers equation can, by the Hopf--Cole transform, be reduced to the linear heat equation with multiplicative potential, and thus solved with the Feynman--Kac formula that in turn can be interpreted as averaging with respect to a polymer measure. In the zero-viscosity case, the Burgers equation can be solved by a variational Hopf--Lax--Oleinik--Hamilton--Jacobi--Bellman (HLOHJB) principle
that can be derived from the large deviation principle for random walk or Brownian motion, see
\cite{FW:MR2953753}. 
As viscosity tends to zero, the polymer measure naturally arising in the Feynman--Kac formula concentrates around paths that minimize action in the HLOHJB variational principle, in precise agreement with zero-temperature limit for finite-volume polymer measures.

The long-term dynamics of the Burgers equation with kick forcing (where a delta-type random force is applied at every integer time and there is no forcing between those kicks) in both positive and zero visosity settings is governed by global stationary solutions whose construction and properties was given in~\cite{kickb:bakhtin2016} and~\cite{2016arXiv160704864B}. It turns out that for each value of the average velocity and almost every realization of the random forcing there is a uniquely compatible global solution that can be seen as a one-point attractor. This statement is often referred to as One Force --- One Solution (1F1S) principle, or synchronization by noise.

The key to understanding 1F1S principle for the Burgers equation is the analysis of polymer measures or action minimizers
over long time intervals. A crucial point is the construction of global solutions using the infinite volume
polymer measures (in the positive viscosity case) or one-sided infinite action minimizers (for zero viscosity).
Another crucial point is to make sense of differences in action (resp.\ free energy) of two infinite one-sided  minimizers (resp.\ polymers). This is done rigorously through a limiting procedure 
leading to the notion of Busemann function.

The ergodic program for the Burgers equation has a long history. Before~\cite{kickb:bakhtin2016} and~\cite{2016arXiv160704864B}, the ideas around 1F1S for Burgers equation (and its generalizations) 
 with random forcing
were explored first in compact setting 
\cite{Sinai:MR1117645},  \cite{ekms:MR1779561}, 
\cite{Iturriaga:MR1952472}, \cite{Gomes-Iturriaga-Khanin:MR2241814},  \cite{Dirr-Souganidis:MR2191776},
  \cite{yb:MR2299503}, \cite{Debussche-Vovelle:MR3418750},
in quasi-compact setting in
\cite{Khanin-Hoang:MR1975784}, \cite{Suidan:MR2141893}, 
 \cite{Bakhtin-quasicompact}, 
 and, finally, in fully noncompact setting in~\cite{BCK:MR3110798},
where stationary Poissonian forcing was considered. 
 The work in~\cite{BCK:MR3110798} used ideas from \cite{Kesten:MR1221154}, \cite{HoNe3}, \cite{HoNe2}, \cite{HoNe},   \cite{Wu}, \cite{CaPi}, \cite{CaPi-ptrf}.
A similar approach to global solutions based on Busemann functions for lattice models was also developed in~\cite{Rassoul-Agha--Seppalainen--Yilmaz:2013arXiv1311.3016G},
\cite{Rassoul-Agha--Seppalainen--Yilmaz:2015arXiv151000859G}.

In \cite{Gomes-Iturriaga-Khanin:MR2241814}, the zero-viscosity limit for stationary solutions of the randomly forced Burgers equation (and other stochastic HJB equations) was obtained in the (compact) case of the circle or torus. In the present paper, we use the zero-temperature limit for infinite-volume directed polymers in order to obtain the zero-viscosity limit for stationary solutions of the  Burgers
equation with random kick forcing. Namely, we prove that as the viscosity vanishes the stationary solutions of the viscous Burgers equation converge to those of the inviscid one.
Of course, the PDE results of~\cite{Gomes-Iturriaga-Khanin:MR2241814}  can also be restated in the polymer language. 

We postpone the precise description of the mode of convergence of global solutions to the later sections of the paper. Here, we only want to make a comment that our results seem to be first ones on conservation of stationary solutions of a nonlinear stochastic PDE in noncompact setting under a transition to a limit. Among hard problems in this direction is the inviscid limit of the stochastic two-dimensional Navier--Stokes system  (SNS).  The compact case such as SNS on the 2D-torus is well understood, see 
\cite{E-Mattingly-Sinai:MR1868992}, \cite{Bricmont-et-al:MR1868991},
\cite{Kuksin-Shirikyan:MR1785459}, \cite{HM:MR2259251}, \cite{HM:MR2478676}, \cite{HM:MR2786645}. However, as the viscosity tends to zero, one needs to scale the forcing appropriately to obtain nontrivial behavior in the limit, as was realized in~\cite{Kuksin-2004:MR2070104},\cite{Kuksin-2007:MR2433681}, and~\cite{Kuksin-2008:MR2448135}. This contradicts the Kraichnan theory of 2D turbulence whose predictions can be interpreted as existence of a nontrivial inviscid limit
under viscosity-independent forcing. This discrepancy can be explained by finite size effects since
the inverse cascades of Kraichnan theory are impossible in a compact domain. 
It would be extremely interesting to see if 
 this contradiction gets resolved in noncompact setting. However, the only ergodic result for
Navier--Stokes system in the entire space known to us is~\cite{Bakhtin:MR2203796}, where 
under certain conditions on the decay of the noise at infinity,
 a unique invariant distribution
on the Le~Jan--Sznitman existence-uniqueness class is constructed for SNS in 
$\R^3$, and this class of solutions neither allows for spatial stationarity nor survives the inviscid limit. 
 
In the present paper, we show that in the Burgers turbulence which exhibits
a lot of contraction compared to the chaotic unstable behavior typical for the true turbulence,
the situation is quite nice and the expected inviscid limit holds. We also conjecture that similar results
hold for more general HJB equations with convex Hamiltonians and appropriately defined polymer models.

The rest of the paper is organized as follows. 
The setting and minimal background from~\cite{kickb:bakhtin2016} and~\cite{2016arXiv160704864B} that we  
we need to state our results are given in Sections~\ref{sec:Burgers} and~\ref{sec:polymers}: in Section~\ref{sec:Burgers} we introduce the relevant information on the Burgers equation, and in Section~\ref{sec:polymers}, we discuss polymers and action minimizers. We state our main results
in Section~\ref{sec:main-results}. In Sections~\ref{sec:properties-of-partition-function}, we remind
some basic useful results on partition functions.
The proofs of the main results are given in Sections~\ref{sec:concentration}---\ref{sec:zero-temperature-limit}.

{\bf Acknowledgments.} YB gratefully acknowledges partial support from NSF through grant
DMS-1460595.  

 \section{The Setting. Burgers equation}\label{sec:Burgers}

 \subsection{Forward and backward Burgers equation}
The one-dimensional Burgers equation describing evolution of a velocity field $u(t,x)$, where
$t\in\R$ and $x\in\R$ are time and space variables, is
 \begin{equation}
\label{eq:forward-Burgers}
 \partial_t u + u \partial_x u=\visc\partial_{xx} u + f.
\end{equation}
Here $f=f(t,x)$ is external forcing, and $\temp\ge0$ is the viscosity parameter. This equation, with
random kick-forcing $f$ was studied in~\cite{kickb:bakhtin2016} for $\temp=0$ and in \cite{2016arXiv160704864B}  for $\temp>0$. To solve the Cauchy problem for this equation up to time $t\in\R$, one needs to emit 
action minimizers and polymers from time $t$ into the past, and this is what was done 
in~\cite{kickb:bakhtin2016} and~\cite{2016arXiv160704864B}. However, it is slightly more natural to work with forward polymers and action minimizers, so in this paper, we change the direction of time 
and state our results for the following ``backward'' Burgers equation in 1D: 
\begin{equation}
\label{eq:Burgers}
- \partial_t u + u \partial_x u= \visc\partial_{xx} u + f.
\end{equation}
For this equation, instead of the initial value problem, it is the terminal value problem that is well-posed. It is natural to solve~\eqref{eq:Burgers} backward in time, and 
if $s>t$, then $u(t,\cdot)$ is uniquely defined by $u(s,\cdot)$ and the forcing $f$ between $t$ and $s$.
We stress that we change the time direction in the Burgers equation just for convenience. Restating any result obtained for equation \eqref{eq:forward-Burgers} in terms of equation~\eqref{eq:Burgers} and {\it vice versa} is trivial.

The Burgers equation is tightly connected to the  following (backward) Hamilton--Jacobi--Bellman (HJB) equation:
\begin{equation}
\label{eq:HJB}
- \partial_t U +  \frac{(\partial_x U)^2}{2}=\visc\partial_{xx} U + F.
\end{equation}
Namely, if $U$ is a solution of~\eqref{eq:HJB}, then $u=\partial_x U$ solves~\eqref{eq:Burgers} with $f=\partial_x F$.

The main model that we study in this paper is the Burgers equation with kick forcing of the following form: 
\begin{equation*}
f(t,x)=\sum_{n\in\Z} f_{n}(x)\delta_{n}(t). 
\end{equation*}

This means that the additive forcing is applied only at integer times. On each interval $(n, n+1]$
where~$n\in\Z$,
the velocity field evolves (from time $n+1$ to time~$n$) according to the unforced backward Burgers equation
\begin{equation}
\label{eq:unforced-Burgers}
-\partial_t u + u \partial_x u=\visc\partial_{xx} u,
\end{equation}
and at time~$n$,
the entire velocity profile~$u$ receives an instantaneous macroscopic increment equal to $f_n$:
\begin{equation}
\label{eq:effect-of-one-kick}
u(n-0,x)=u(n,x)+f_{n}(x),\quad x\in \R.
\end{equation}
We assume that the potential $F=F_{n,\omega}(x)$ of the forcing 
\[
f_n(x)=f_{n,\omega}(x)=\partial_x F_{n,\omega}(x), \quad n\in\Z,\ x\in\R,\ \omega\in \Omega, 
\]
is a stationary random field defined on some probability space $(\Omega,\Fc,\Pp)$.
We will describe all conditions that we impose on $F$ in section~\ref{sec:assumption-on-F}. At this point, we need only the following consequence of those
conditions: for every $\omega\in\Omega$
and every~$n\in\Z$, the function $F_{n,\omega}(\cdot)$ is measurable with respect to $\omega$, continuous with respect to~$x$,  and satisfies  
\begin{equation}
\label{eq:forcing-averages-to-0}
\lim_{|x|\to\infty} \frac{F_{n,\omega}(x)}{|x|}=0.
\end{equation}

\bigskip

Let us now explain how to solve the backward Burgers dynamics with kick forcing, thus introducing the
dynamics that we will study in this paper.
The inviscid case ($\temp = 0$) and the viscous case ($\temp > 0$) will be treated separately.
For the viscous case, the Burgers dynamics will be defined through  the Hopf--Cole
transformation and the Feynman--Kac formula.
For the inviscid case, we will use a variational characterization that can be seen as the 
limiting case of the positive viscosity formula.

For every $m,n\in\Z$ satisfying $m<n$, we denote the set of all paths~
\begin{equation*}
\gamma: [m,n]_{\Z} = \{m,m+1,\ldots,n\}\to\R
\end{equation*}
by $S_{*,*}^{m,n}$.
If in addition a point $x\in\R$ is given, then $S_{x,*}^{m,n}$ denotes the set of all such paths that
satisfy $\gamma_m=x$. If $n=\infty$, then we understand the above spaces as the spaces of one-sided semi-infinite paths.
If points $x,y\in\R$ are given, then $S_{x,y}^{m,n}$  denotes the set of all such paths that
satisfy $\gamma_m=x$ and $\gamma_{n}=y$. 

Let $m < n$.  Given a path $\gamma$ defined on $[m',n']_{\Z} \supset [m,n]_{\Z}$, its kinetic energy~$I^{m,n}(\gamma)$, potential energy~$H^{m,n}_{\omega}(\gamma)$ and total action~$A^{m,n}_{\omega}(\gamma)$ are given by 
\begin{equation}
  \label{eq:def-of-action}
  \begin{gathered}
  I^{m,n}(\gamma) = \frac{1}{2}\sum_{k=m+1}^n (\gamma_k - \gamma_{k-1})^2, \quad
  H^{m,n}_{\omega}(\gamma) = \sum_{k=m+1}^n F_{k,\omega}(\gamma_k),\\
  A^{m,n}_{\omega}(\gamma) = I^{m,n}(\gamma) + H^{m,n}_{\omega}(\gamma).    
  \end{gathered}
\end{equation}
Note the asymmetry in the definition of $H_\omega^{m,n}$: we have to include $k=n$, but exclude
$k=m$.
All our results
are proved for this choice of path energy, but it is straightforward to obtain their counterparts for the version
of energy where the~$k=n$ is excluded and~$k=m$ is included. 
For the inviscid case, we can now define the random backward evolution operator on potential by 
\begin{equation}
\label{eq:inviscid-evolution-on-potential}
[\Psi^{m,n}_{0, \omega} U](x) = \inf_{\gamma \in S^{m,n}_{x,*}} \{  U \big( \gamma_n \big) +
A^{m,n}(\gamma) \}, \quad x \in \R, \ m < n.
\end{equation}
For the viscous case, one can introduce the Hopf--Cole transformation~$\varphi$ by
\begin{equation}
\label{eq:Hopf-Cole-2}
 \varphi(t,x)=e^{- \frac{U(t,x)}{\temp}}.
\end{equation}
An application of the discrete Feynman--Kac formula will lead to the following backward evolution operator on $\varphi$:
\begin{equation}
\label{eq:F-K}
[ \Xi^{m,n}_{\temp, \omega} \varphi ](x)=\int_{\R} \hat{Z}^{m,n}_{x,y; \temp, \omega}\varphi(y)\, dx,\quad
x\in\R,\  m < n, 
\end{equation}
where
\begin{multline}
\label{eq:Z-hat}
\hat{Z}^{m,n}_{x,y; \temp, \omega} \\
=\int_{\R}\dots\int_\R \prod_{k=m+1}^n 
\left[g_{\temp}(x_k-x_{k-1})e^{-\frac{F_k(x_k)}{\temp}}\right] \delta_x(dx_m) dx_{m+1}\ldots dx_{n-1}\delta_y(dx_n)
\end{multline}
and $g_{\kappa}(x) = \frac{1}{\sqrt{2 \pi \kappa}} e^{-\frac{x^2 }{2\kappa}}$.
With the inverse of the Hopf--Cole transform~\eqref{eq:Hopf-Cole-2}, we can define evolution on potentials by
\[
\Phi_{\temp, \omega}^{m,n} U = -\temp\ln \Xi_{\temp, \omega}^{m,n} e^{- \frac{U}{\temp}}. 
\]

The space of velocity potentials that we will consider will be~$\HH$, the space of all locally
Lipschitz functions $W:\R\to\R$
 satisfying
\begin{align*}
 \liminf_{x\to\pm\infty}\frac{W(x)}{|x|}&>-\infty.
\end{align*}
We will also need a family of spaces
\[
\HH(v_-,v_+)=\left\{W\in\HH:\ \lim_{x\to \pm\infty} \frac{W(x)}{x}=v_\pm \right\},\quad v_-,v_+\in\R.
\] 

\begin{lemma} \label{lem:invariant_spaces}For every~$\temp \ge 0$ and any $\omega\in\Omega$,
for any $l,n,m\in\Z$ with $l<n<m$ and $W\in\HH$,
\begin{enumerate}
 \item\label{it:inv-of-HH} $\Phi^{n,m}_{\temp, \omega} W$ is well-defined and belongs to $\HH$;
 \item\label{it:inv-of-HH-vv} if $W\in\HH(v_-,v_+)$ for some  $v_-,v_+$, then $\Phi^{n,m}_{\temp, \omega} W\in\HH(v_-,v_+)$;
 \item\label{it:cocycle} (cocycle property) $\Phi^{l,m}_{\temp, \omega}W=\Phi^{l, n}_{\temp,\omega}\Phi^{n, m}_{\temp, \omega}W$.  
\end{enumerate}
\end{lemma}
We can also introduce the Burgers dynamics on the space $\HH'$ of velocities $w$ such that
for some function $W\in\HH$ and Lebesgue almost every~$x$,  $w(x)=W'(x)=\partial_x W(x)$.
For all $v_-,v_+\in\R$,  $\HH'(v_-,v_+)$ is the space of velocity profile with well-defined one-sided averages $v_-$ and $v_+$, it consists of functions $w$ such that
the potential~$W$ defined by $W(x)=\int_0^x w(y)dy$ belongs to $\HH(v_-,v_+)$.

We will
write $ w_1=\Psi^{n_0,n_1}_{\temp, \omega} w_0$ if $w_0=W'_0$, $w_1=W'_1$, and
$W_1=\Phi^{n_0,n_1}_{\temp, \omega}
W_0$  for some $W_0,W_1\in\HH$. 

\subsection{Assumptions on the random forcing}
\label{sec:assumption-on-F}
 For simplicity, we will work on the canonical probability space $(\Omega_0,\Fc_0,\Pp_0)$ 
of realizations of the potential, although other more general settings are also possible. 
We assume that $\Omega_0$ is the space of continuous functions $F:\RZ\to\R$ equipped with $\Fc_0$, the completion of the Borel $\sigma$-algebra with respect to local uniform topology, and $\Pp_0$ is
a probability measure preserved by the group of shifts $(\theta^{n,x})_{(n,x)\in\ZR}$  defined by
\[
 (\theta^{n,x}F)_m(y) = F_{n+m}(x+y),\quad (n,x),(m,y)\in\ZR, 
\]
i.e., $(F_n(x))_{(n,x)\in\ZR}$ is a space-time stationary process. In this framework,
$F=F_\omega=\omega$, and we will use all these notations intermittently.

In addition to this, we introduce the following requirements: 
\begin{description}
\item[{(A1)}\label{item:stationary-in-space}] The flow $(\theta^{0,x})_{x\in\R}$ is ergodic. In particular,
for every $n\in\Z$,  $F_n(\cdot)$ is ergodic with respect to the spatial shifts.
\item[{(A2)}\label{item:indep-in-time}] The sequence of processes $\big( F_n(\cdot)\big)_{n \in \Z}$
  is \iid
\item[(A3)\label{item:smooth}] With probability 1, for all $n\in\Z$, $F_n(\cdot)\in C^1(\R)$. 
\item[(A4)\label{item:exponential-moment-with-viscosity}] For all  $(n,x) \in \Z\times\R$ and all
  $\beta \in \R_+$,
  \begin{equation*}
    \lambda(\beta) :=\E e^{- \beta F_n(x)} < \infty.
  \end{equation*}
\item[(A5)\label{item:exponential-moment-for-maximum}] 
There are~$\varphi$, $\eta>0$ such that for all $(n,j) \in \Z  \times \Z$,
  \begin{equation*}
e^{\varphi} =   \E e^{\eta F^{*}_{n,\omega}(j)} < \infty,
\end{equation*}
where
\begin{equation}
\label{eq:def-F-star}
F^{*}_{n,\omega}(j) = \sup \{ |F_{n,\omega}(x)| : x \in [j,j+1] \}.
\end{equation}
\end{description}

We will use these standing assumptions throughout the paper.

Stationarity and~\ref{item:exponential-moment-for-maximum} imply that~\eqref{eq:forcing-averages-to-0} holds
with probability~$1$ on $\Omega_0$. It will be convenient in this paper to work on a modified probability space
\begin{equation}
\label{eq:def-of-Omega}
\Omega=\left\{F\in\Omega_0: \lim_{|x|\to\infty}\frac{F_{n}(x)}{|x|}=0,\quad n\in\Z\right\}\in\Fc_0.
\end{equation}
of probability $1$ instead of $\Omega_0$. On this set, the Burgers evolution possesses 
some nice properties discussed in~\cite{kickb:bakhtin2016} and~\cite{2016arXiv160704864B}.
Moreover, $\Omega$ is invariant 
under space-time shifts~$\theta^{n,x}$ and under Galilean  space-time shear 
transformations~$L^{v},$ $v\in\R$,
defined by 
\begin{equation*}
(L^v F)_n (x)=F_n(x+vn),\quad (n,x)\in\ZR. 
\end{equation*}

We denote the restrictions of $\Fc_0$ and $\Pp_0$ onto $\Omega$ by $\Fc$ and $\Pp$.
From
now on we work with the probability space $(\Omega,\Fc,\Pp)$. Under this modification, all the
distributional properties are preserved.

\section{Directed polymers and minimizers}\label{sec:polymers}

Formulas \eqref{eq:inviscid-evolution-on-potential} and \eqref{eq:F-K}--\eqref{eq:Z-hat} show that
the problem of long-term properties of the Burgers equation with random forcing can be approached
through analysis of properties of either action minimizing paths (for the inviscid case) 
or Gibbs distributions on paths (for the viscous case) over long time intervals. This section
summarizes the results of~\cite{kickb:bakhtin2016} and~\cite{2016arXiv160704864B} for both settings.
We first describe properties of finite and one-sided infinite  minimizers in Section~\ref{sec:minimizers}, then the same is done for
polymers and their thermodynamic limits in Section\ref{sec:polymer-measures}, and finally we stress the connection to the global solutions
of the Burgers equation in Section~\ref{sec:connection-to-Burgers}.

\subsection{Minimizers}
\label{sec:minimizers}
For 
every $(m,x)\in\ZR$ and every $v\in\R$,  we denote
\[
S_{x,*}^{m,+\infty}(v)=\left\{\gamma\in S_{x,*}^{m,+\infty}:\ \lim_{n\to\infty}\frac{\gamma_n}{n}=v\right\}.
\]
If $\gamma \in S_{x,*}^{m,+\infty}(v)$, then we say that $\gamma$ has asymptotic slope $v$.

Let $A^{m,n}_{x,y}= A^{m,n}(x,y)$ denote the minimal action between~$(m,x)$ and~$(n,y)$, that is, 
\begin{equation}
  \label{eq:def-p2p-action}
A^{m,n}(x,y) = \min_{\gamma \in S^{m,n}_{x,y}} A^{m,n}(\gamma).
\end{equation}

A~path $\gamma \in S_{*,*}^{m,n}$ is called a (finite) minimizer if~$A^{m,n}(\gamma) = A^{m,n}(\gamma_m,
\gamma_n)$.  A~path~${\gamma \in S_{*,*}^{m,+\infty}}$ is called a semi-infinite minimizer (or
simply minimizer if it is clear from the context) if for any $n_2>n_1 > m$,
$\gamma^{n_1, n_2}$ is
a minimizer, where~$\gamma^{n_1,n_2}$ denotes the restriction of~$\gamma$ to~$[n_1, n_2]_{\Z}$.

The following theorem summarizes the results on semi-infinite minimizers established in~\cite{kickb:bakhtin2016}. These results were established in~\cite{kickb:bakhtin2016}
for a specific random potential of shot-noise type, but it is easy
to see that they hold true for any potential satisfying 
assumptions~\ref{item:stationary-in-space}--\ref{item:exponential-moment-for-maximum} under the additional requirement of finite dependence range. It is also natural to expect that they hold for
a much broader class of mixing potentials.

\begin{theorem}[Theorem 3.3, Lemma 9.3 in~\cite{kickb:bakhtin2016}]
\label{thm:zero-temperature-infinite-minimizers} 
Suppose that assumptions \ref{item:stationary-in-space}--\ref{item:exponential-moment-for-maximum}
are satisfied and $F$ has finite dependence range. Then for every $v \in \R$,
there is a full measure set $\Omega_{v,0}$ such that the following properties hold:
\begin{enumerate}[1.]
\item\label{item:existence-of-minimizer}
 For all $\omega \in
\Omega_{v,0}$, there is an at most countable set $\mathcal{N}=\Nc_{\omega} \in \ZR$ such that for all $(m,x)
\in \Z\times\R \setminus \Nc$, there is a unique minimizer $\gamma_x^{n,+\infty}(v) \in S^{m,+\infty}_{x,*}(v)$.
\item\label{item:def-busemann-function} (Busemann function) Let~$\omega \in \Omega_{v,0}$.
  For $(n_1,x_1), (n_2,x_2) \in \Z\times\R$, there is sequence $N_k \uparrow +\infty$ such that the
  limit 
\begin{equation}\label{eq:def-busemann}
B_v \big(  (n_1,x_1), (n_2,x_2) \big) = \lim_{k \to \infty} A^{n_1, N_k} \big(
\gamma_{x_1}^{n_1}(v) \big) - A^{n_2,N_k} \big(  \gamma_{x_2}^{n_2}(v) \big)
\end{equation}
exists.  Here, if the semi-infinite minimizer is not unique at~$(n_i,x_i)$,
then~$\gamma_{x_i}^{n_i}(v)$ can be any minimizer in $S^{n_i,\infty}_{x_i, *}(v)$, $i
= 1,2$.
Moreover, if the limit in~(\ref{eq:def-busemann}) exists for some other sequence $(N'_k)$, then it is
independent of the choice of~$(N'_k)$.
The function $B_v$ has the property that for any $(n_i, x_i) \in \ZR$,
\begin{gather}\label{eq:cocycle-property-for-zero-temp-busemann}
  B_v \big(  (n_1,x_1), (n_2,x_2) \big) + B_v \big(  (n_2, x_2), (n_3, x_3) \big)
  = B_v \big(  (n_1,x_1), (n_3, x_3) \big), \\ 
\nonumber  B_v \big(  (n_1,x_1), (n_2, x_2) \big) = -B_v \big(  (n_2,x_2), (n_1,x_1) \big).
\end{gather}
\item\label{item:relation-to-inviscid-burgers} 
  The function $U_{v; 0}(n, \cdot) = -  B_v \big(  (n,\cdot), (n,0) \big)$ is Lipschitz, and it is differentiable
  at all~$(n,x) \not\in \Nc$.  The derivative is given by 
\begin{equation}\label{derivative-of-U-n-zero-temperature}
u_{v; 0}(n, x):= \frac{d}{dx} U_{v; 0}(x) =x  - \big( \gamma_x^{n,+\infty}(v)  \big)_{n+1}.
\end{equation}
\item\label{item:Bussmann-function-solve-variational-problem}(Solution to inviscid Burgers and HJB equations)  The
  function $B_v$ solves the following variational problem: for $m > n$ and fixed $(n_0, x_0) \in \ZR$,
\begin{equation}\label{eq:busemann-function-solving-invisid-burgers}
  B_v \big(  (n,x), (n_0, x_0) \big)
  = \min_{ y \in \R} \{  B_v \big(  (m,y), (n_0, x_0) \big)  + A^{n,m}(x,y) \}.
\end{equation}
In particular, the function $u_{v;0}$ introduced in~\eqref{derivative-of-U-n-zero-temperature} solves the inviscid
Burgers equation.
\end{enumerate}
\end{theorem}

\subsection{Polymer measures}
\label{sec:polymer-measures}
Let $\temp > 0$. In the context of polymer measures, this parameter plays the role of temperature.
For $m,n\in\Z$ with~$m < n$ and~$x,y\in\R$, 
the point-to-point polymer measure (at temperature~$\temp$)~$\mu^{m,n}_{x,y; \temp, \omega}$ is a
probability measure on $S^{m,n}_{x,y}$ that has density
\begin{align*}
\mu^{m,n}_{x,y; \temp, \omega}(x_m,\ldots,x_n)
&=\frac{\prod_{k=m+1}^n \left[g_{\temp}(x_k-x_{k-1})e^{-\frac{F_k(x_k)}{\temp}}\right]
}{\hat{Z}^{m,n}_{x,y; \temp, \omega}},
\end{align*}
with respect to~${\delta_x\times \Leb^{n-m-1}\times \delta_y}$, where~$\hat{Z}^{m,n}_{x,y;
  \temp,\omega}$ is defined in~(\ref{eq:Z-hat}).

Let us introduce
\begin{align}\notag
    Z^{m,n}_{x,y; \temp, \omega} &= \big( 2\pi \temp \big)^{n/2} \hat{Z}^{m,n}_{x,y; \temp, \omega}
    =\int_{\gamma \in S^{m,n}_{x,y}} e^{- \be   A_{\omega}^{m,n}(\gamma) }  \, d \gamma \\
    & =  \int e^{- \be \sum\limits_{k = m+1}^{n} \big[ \frac{1}{2}(x_k - x_{k-1})^2 + F_k(x_k)  \big]} \,
    \delta_x(dx_m) dx_{m+1} ... dx_{n-1} \delta(dx_n),   
    \label{eq:Z}
\end{align}
where $A^{m,n}$ is defined in~(\ref{eq:def-of-action}).
The polymer density can also be expressed as
\[
 \mu^{m,n}_{x,y; \temp,\omega}(\gamma_m,\ldots,\gamma_n)= \frac{
 e^{- \be A_\omega^{m,n}(\gamma)}}{Z^{m,n}_{x,y; \temp, \omega}}.
\]
We often omit the~$\omega$
argument in all the notations used above. We also often write~$Z^{m,n}_{\temp}(x,y)$ for $Z^{m,n}_{x,y; \temp}$.

We call a measure $\mu$ on  $S^{m,n}_{x,*}$ a polymer measure (at temperature $\temp$) if there is a probability measure~$\nu$ on~$\R$
such that $\mu=\mu_{x,\nu; \temp}^{m,n}$, where 
\[
\mu_{x,\nu; \temp}^{m,n}=\int_{\R}\mu_{x,y; \temp}^{m,n}\nu(dy). 
\] 
We call $\nu$ the terminal measure for~$\mu=\mu_{x,\nu; \temp}^{m,n}$. It is also natural to call $\mu$ a point-to-measure polymer
measure associated to~$x$ and~$\nu$. 

A measure $\mu$ on $S_{x}^{m,+\infty}$ is called an infinite volume polymer measure if for any $n\ge m$
the projection of $\mu$ on $S^{m,n}_{x,*}$ is a polymer measure. This condition is equivalent to the Dobrushin--Lanford--Ruelle (DLR) condition
on the measure $\mu$.

We say that the strong law of large numbers (SLLN) with slope $v\in\R$ holds for a measure $\mu$ on $S_{x,*}^{m,+\infty}$ if $\mu(S_{x,*}^{m,+\infty}(v))=1$. 

We say that LLN with slope $v\in\R$ holds for a sequence of Borel measures $(\nu_n)$ on $\R$ if
for all $\delta>0$,
\[
 \lim_{n\to\infty} \nu_n([(v-\delta)n,(v+\delta) n])=1.
\]
Finally, for any $(m,x)\in\ZR$, we say that a measure $\mu$ on $S_{x,*}^{m,+\infty}$ satisfies LLN with slope $v$ if the sequence of its marginals 
$\nu_k(\cdot)=\mu\{\gamma:\ \gamma_k\in\cdot\}$ does.

We denote by $\Pc_{x;\temp}^{m,+\infty}(v)$  the set of all polymer measures at temperature~$\temp$
on $S_{x,*}^{m,+\infty}$ satisfying SLLN with slope $v$. The set of all polymer measures at temperature~$\temp$ on $S_{x,*}^{m,+\infty}$ satisfying LLN with slope $v$ 
is denoted by $\widetilde \Pc_{x;\temp}^{m,+\infty}(v)$. These sets are random since they depend on the realization of the environment, but we suppress the dependence on $\omega\in\Omega$ in this notation.

The following theorem summarizes the results established in~\cite{2016arXiv160704864B} on the
infinite polymer measure with given asymptotic slope.
\begin{theorem}[Theorems 4.2, 4.3, 11.2 in~\cite{2016arXiv160704864B}]\label{thm:positive-temperature-IVPM}
Suppose that assumptions~\ref{item:stationary-in-space}--\ref{item:exponential-moment-for-maximum}
are satisfied. Then, for each $v\in\R$ and~$\temp > 0$, there is a full measure set $\Omega_{v,\temp}\in\Fc$ such that
\begin{enumerate}[1.]
 \item \label{it:existence-uniqueness-of-infinite-volume} For all $\omega\in\Omega_{v,\temp}$ and
   all $(m,x)\in\ZR$, there is a unique polymer measure $\mu_{x; v, \temp}^{m,+\infty}$ such
   that 
   \[\Pc_{x;\temp}^{m,+\infty}(v)=\widetilde\Pc_{x;\temp}^{m,+\infty}(v)=\{\mu_{x; v, \temp}^{m,+\infty}\}.\]
    \item \label{it:thermodynamic-limit} For all $\omega\in\Omega_{v,\temp}$, all $(m,x)\in\Z\times\R$ and 
 for every sequence of measures $(\nu_n)$ satisfying LLN with slope $v$, 
 finite-dimensional distributions of~$\mu_{x,\nu_n; \temp}^{m,n}$ converge to $\mu_{x; v,\temp}^{m,+\infty}$ in
 total variation. 
\item\label{item:ration-of-partition-function}  For all $\omega\in\Omega_{v,\temp}$, all
$(n_1,x_1),(n_2,x_2) \in \ZR$
and for every sequence~$(y_N)$ with~$\lim\limits_{N \to  \infty} y_N/N = v$, we have
\begin{equation*}
\lim\limits_{N\to \infty}  \frac{Z_{x_1, y_N; \temp}^{n_1, N}}{Z_{ x_2, y_N; \temp}^{n_2, N}} = G,
\end{equation*}
where $G=G_{v, \temp} \big(  (n_1,x_1) ,( n_2,x_1) \big) \in(0,\infty)$ does not depend on~$(y_N)$.
Moreover, the function $G$ has the property that for any $(n_i, x_i) \in \ZR$, 
\begin{gather}\label{eq:cocycle-property-for-positive-temp-busemann}
  G_{v,\temp} \big(  (n_1,x_1), (n_2,x_2) \big)  G_{v,\temp} \big(  (n_2,x_2) , (n_3, x_3) \big)
  = G_{v,\temp} \big(  (n_1, x_1), (n_3,x_3) \big), \\ \nonumber
  G_{v,\temp} \big(  (n_1,x_1), (n_2, x_2) \big) = \Big[  G_{v,\temp} \big(  (n_2,x_2), (n_1,x_1) \big) \Big]^{-1}.
\end{gather}

\item\label{item:density-via-partition-function-ratio} 
For all $(m,x) \in \ZR$, the finite-dimensional distributions of $\mu_{x; v, \temp}^{m,+\infty}$ are absolutely continuous.  The
density of its marginal is given by
\begin{equation}
  \label{eq:expression-of-polymer-measure-density}
  \mu_{x;v,  \temp}^{m,+\infty}  \pi_{n}^{-1}   (dy) = Z^{m,n}_{x,y; \temp} G_{v; \temp} \bigl(
  (n,y), (m,x) \bigr),  \quad n > m,
\end{equation}
where $\pi_n$ is the projection of a path $\gamma$ onto its $n$-th coordinate $\gamma_n$.
\item\label{item:relation-to-viscous-burgers} Let $U_{v; \temp}(n, \cdot) = -  \temp \ln G_{v,\temp} \big(  (n,\cdot), (n,0)
  \big)$, then 
\begin{equation}\label{eq:derivative-of-U-n-positive-temperature}
  u_{v; \temp}(n, x) :=   \frac{d}{dx} U_{v;\temp}(n, x) =   \int(x-y) \mu_{x; v, \temp}^{n,+\infty} \pi_{n+1}^{-1}
  (dy), \quad (n,x) \in \ZR.
\end{equation} 
\item\label{item:partition-function-ratios-solve-burgers}(Solutions to viscous Burgers, HJB, and heat equations) The function $G_{v; \temp}$ satisfies the
  following relation: for $m > n$ and fixed $(n_0,x_0) \in \ZR$, 
\begin{equation}\label{eq:partition-function-ration-solves-burges}
  G_{v; \temp} \big(  (n,x), (n_0,x_0) \big)
  = \int_{\R} Z^{n,m}_{x,y;\temp} G_{v;\temp} \big(  (m,y), (n_0,x_0) \big).
\end{equation}
In particular, $ u_{v; \temp}(n,x)$ defined in~(\ref{eq:derivative-of-U-n-positive-temperature}) solves
Burgers equation with viscosity~$\temp$.
\end{enumerate}
\end{theorem}

\subsection{Connections to global solutions of Burgers equation}
\label{sec:connection-to-Burgers}
We say that ${u(n,x)=u_\omega(n,x)}$, $(n,x)\in\Z\times\R$ is a global solution for the
Burgers equation with viscosity $\temp$ if there is a set $\Omega'\in\Fc$ with $\Pp(\Omega')=1$ such that for all
$\omega\in\Omega'$, all~$m$ and $n$ with $m<n$, we have $\Psi^{m,n}_{\temp, \omega} u_\omega(n,\cdot)= u_\omega(m,\cdot)$.

 We recall the full measure sets $\Omega_{v, \temp}$ and the
functions $u_{v;\temp }$, $v \in\R$,
$\temp \ge 0$ defined in Theorem~\ref{thm:zero-temperature-infinite-minimizers}
and~\ref{thm:positive-temperature-IVPM}.
As we see in the previous two sections, the relations~(\ref{eq:busemann-function-solving-invisid-burgers})
and~(\ref{eq:partition-function-ration-solves-burges}), together
with~(\ref{derivative-of-U-n-zero-temperature})
and~(\ref{eq:derivative-of-U-n-positive-temperature}) where $u_{v; \temp}$ are defined, show that for each $\temp \ge 0$, $u_{v; \temp}$ is a global
solution for the Burgers with viscosity $\temp$.  In fact, they are the only ones in a certain sense,
as the following theorem states.
\begin{theorem}[\cite{kickb:bakhtin2016}, \cite{2016arXiv160704864B}]
\label{thm:global-solution}
Let $\temp \ge 0$.
For every $v \in \R$, the function $u_{v;\temp}$ defined on the full measure set~$\Omega_{v,
  \temp}$ is a unique global stationary solution in $\HH'(v,v)$ for the Burgers equation with viscosity $\temp$.
\end{theorem}

\section{Main results}\label{sec:main-results}
In this section, we state the main results of this paper.  
Our first result concerns the zero-temperature limit of infinite volume polymer measures:
\begin{theorem}
  \label{thm:invisid-limit-for-polymer-measures}
Let $v \in \R$. With probability one, the following holds true:
\begin{enumerate}[1.]
\item\label{item:existence-of-inifinte-polymer-measures}
For all $v \in \R$,  all $\temp \in (0,1]$ and all $(m,x) \in \ZR$, $\mathcal{P}^{m,+\infty}_{x;\temp}(v) \neq \varnothing$.
\item\label{item:tightness-of-polymer-measures} Let $v \in \R$ and~ $(m,x) \in \ZR$.  Then  
the family of probability measures $ (\Pc^{m,+\infty}_{x;\temp}(v))_{\temp
    \in (0,1]}$ 
  on~$S^{m,+\infty}_{x,*} \cong \R^{\N}$ is tight.
\item \label{item:convergence-of-polymer-measures} {\bf (Zero-temperature limit.)}
For fixed $v \in \R$ and~$(m,x) \in \ZR$, let~$\mu_{\temp} \in \Pc^{m,+\infty}_{x;\temp}(v)$, $\temp \in (0,1]$.
Then, any limit point $\mu$ of $\big( \mu_{\temp} \big)$ as~$\temp \downarrow 0$ concentrates on semi-infinite
minimizers on $S^{m,+\infty}_{x,*}(v)$.
In particular, if $S^{m,+\infty}_{x,*}$ contains only one element $\gamma$, then $\mu$ is the $\delta$-measure on $\gamma$.
\end{enumerate}
\end{theorem}

Given $v \in \R$,   Theorem~\ref{thm:positive-temperature-IVPM} says that at every fixed temperature~$\temp>0$, there is a full measure set $\Omega_{v; \temp}$ on which $\Pc^{m,+\infty}_{x,\temp}(v)$ contains a
unique element.  However,  we cannot guarantee the existence of a common full measure set
on which this holds for all values of~$\temp$ simultaneously.
Nevertheless,  in Theorem~\ref{thm:invisid-limit-for-polymer-measures}, using a compactness argument we are able to find a full measure set
on which $\Pc^{m,+\infty}_{x,\temp}(v)$ is always nonempty for all $v \in
\R$, but may potentially contain more than one element.
If one considers only countably many values of temperatures, then this difficulty with common
exceptional sets does not arise. This approach is used in the next result.

Let us now state our main theorem on  the inviscid limit of the global solutions of Burgers equation.
In addition to~\ref{item:stationary-in-space}--\ref{item:exponential-moment-for-maximum}, in this
section we also assume the potential $F$ has the property such that
conclusions of Theorem~\ref{thm:zero-temperature-infinite-minimizers} hold true (see the discussion before
Theorem~\ref{thm:zero-temperature-infinite-minimizers}), so that the global solution for inviscid
Burgers is unique.
To state this result, we need to specify the topology in which the solutions converge.
We recall that in the kick forcing case, if $u(n,x)$ is a solution to the Burgers equation with
viscosity~$\temp\ge 0$,  then $x-u(n,x)$ is a monotone increasing function (see Lemma~2.1
in~\cite{kickb:bakhtin2016} and Lemma 2.2 in~\cite{2016arXiv160704864B}).
For this reason, it is natural to consider the space $\GG$  of {\it cadlag} (i.e., right-continuous with left limits) functions $u$ such that $x-u(x)$ is increasing. The monotonicity allows to define $\GG$-convergence of a sequence of functions $u_n\in\GG$
to a function $u\in\GG$  as convergence $u_n(x)\to u(x)$, $n\to\infty$, for every continuity point
$x$ of $u$.
The space $\GG$ was first introduced in~\cite{kickb:bakhtin2016}. 
\begin{theorem} 
  \label{thm:convergence-of-busemann-function}
  Let $v\in\R$ and fix a countable set $\mathcal{D} \subset (0,1]$ that has $0$ as its limit point.
  Then there exists a full measure set~$\hat{\Omega}_v \subset \Omega_{v;0}\cap \bigcap_{\temp \in
    \mathcal{D}}\Omega_{v;\temp}$ such that the following holds true:
\begin{enumerate}[1.]
\item\label{item:convergence-of-PM-countable-temperature}For every~$(m,x) \not\in \Nc$, as
  $\mathcal{D} \ni \temp \to 0$, $\mu_{x;v,
    \temp}^{m,+\infty}$ converge to $\delta_{\gamma^m_x(v)}$ weakly.
\item \label{item:convergen-of-global-solution} {\bf (Inviscid limit for stationary solutions of the Burgers equation.)}
For every $n\in\Z$, 
  $u_{v; \temp}(n, \cdot)\to u_{v; 0}(n, \cdot)$ in $\GG$ as
  $\mathcal{D} \ni \temp \to 0$, where $u_{v;\temp}$ are the global solutions defined in~\eqref{eq:derivative-of-U-n-positive-temperature} 
  for $\temp>0$ and in~\eqref{derivative-of-U-n-zero-temperature} for $\temp=0$.
\item \label{item:convergence-of-partition-function-ratio}
 {\bf (Inviscid limit for Busemann functions and global solutions of the HJB equation.)}
For all~$(n_1,x_1),(n_2,x_2)\in\ZR$, 
\begin{equation*}
\lim_{\mathcal{D} \ni \temp \to 0} - \temp \ln G_{v; \temp} \big(  (n_1,x_1), (n_2, x_2) \big) = B_v
\big(  (n_1, x_1), (n_2, x_2) \big).
\end{equation*}
\end{enumerate}
\end{theorem}

The proofs of the theorems in this section will be given in Section~\ref{sec:zero-temperature-limit}, after a series of auxiliary results in Sections~\ref{sec:properties-of-partition-function}---\ref{sec:delta-straightness-and-tightness}.

\section{Properties of the partition function} \label{sec:properties-of-partition-function}

In this section, we recall useful results on minimal action and partition functions from 
~\cite{kickb:bakhtin2016} and~\cite{2016arXiv160704864B}.

We begin with a lemma on the behavior of distributional properties
of partition functions under shift and shear transformations of space-time. We write $\disteq$ to
denote identity in distribution. 
\begin{lemma}[Lemma 5.1 in \cite{2016arXiv160704864B}]
\label{lem:invariance} Let $\temp\in(0,1]$.
 For any $m,n\in\Z$ satisfying $m<n$ and any points $x,y\in\R$,
 \[
  Z^{n+l,m+l}_{\temp}(x+\Delta,y+\Delta)\disteq Z_{\temp}^{n,m}(x,y). 
 \]
Also, for any $v\in\R$,
\begin{equation}
\label{eq:7}
 Z_{\temp}^{0,n}(0, vn)\disteq e^{-\be \frac{v^2}{2}n} Z_{\temp}^{0,n}(0,0).
\end{equation}

\end{lemma}

It is easy to extend this lemma to obtain the following:
\begin{lemma} \label{lem:shear-for-Z_v}  
Let $\temp \in (0,1]$ and $Z_{v; \temp}(n)=e^{\frac{1}{\temp}\frac{v^2}{2}n}Z^{0, n}_{\temp}(0, vn)$, $n\in\N$, $v\in\R$.
Then the distribution of the process $Z_{v; \temp}(\cdot)$ does not depend on $v$.
Also, for every $n\in\N$, the process $\bar Z_{n; \temp}(x)=e^{\frac{1}{\temp} \frac{x^2}{2}n
}Z^{0, n}_{\temp}(0, x)$, $x\in\R$, is stationary in $x$.
\end{lemma}

The directional linear growth of $\ln Z_{\temp}^{m,n}(x,y)$ over long
time intervals is given by the following result from Section 6 in~\cite{2016arXiv160704864B}:

\begin{theorem} 
\label{thm:free-energy-limit-nonzero}
  There are constants $\alpha_{0; \temp}\in\R$ such that for any $v\in\R$ and $\temp
  \in (0,1]$,
\begin{equation}\label{eq:shape-theorem}
 \lim_{n\to\infty} \frac{\temp \ln Z_{\temp}^{0,n}(0, vn)}{n}\aseq \alpha_{\temp}(v) := \alpha_{0; \temp}-\frac{v^2}{2}.
\end{equation}
\end{theorem}
The function $\alpha_{\temp}(v)$ is called the shape function or the density of free energy. The existence of the limit in~\eqref{eq:shape-theorem} is based on the sub-additive ergodic theorem, and
the quadratic form of $\alpha_{\temp}(v)$ is due to~(\ref{eq:7}).

The counterparts of Lemma~\ref{lem:invariance} and Theorem~\ref{thm:free-energy-limit-nonzero} 
for the inviscid case
were established in~\cite{kickb:bakhtin2016}.  Let us briefly summarize them.
We recall that $A^{m,n}(x,y)$ defined in~(\ref{eq:def-p2p-action}).
It is easy to see that 
\begin{equation}\label{eq:limit-of-lnZ}
\lim_{\temp \downarrow 0} \temp \ln Z^{m,n}_{\temp} (x,y)=-  A^{m,n}(x,y).
\end{equation}
We have the following:
\begin{theorem}
  \label{thm:properties-of-optimal-action}
  \begin{enumerate}
  \item  For any $l \in \Z$ and $\Delta \in \R$, $$A^{m+l, n+l} (x + \Delta, y + \Delta)  \disteq A^{m,n}(x,y).$$
\item For any $v \in \R$, $-A^{0, n}(0, vn) \disteq -A^{0,n}(0,0) - \frac{v^2}{2}n$.
\item   There is a constant $\alpha_{0,0}\in\R$ such that for any $v\in\R$,
\[
 \lim_{n\to\infty} \frac{ -A^{0,n}(0,vn)}{n}\aseq \alpha_{0}(v) := \alpha_{0,0}-\frac{v^2}{2}.
\]
\end{enumerate}
\end{theorem}

It is natural to define  
\begin{equation}\label{eq:def-of-finite-free-energy}
p_n(\temp) =
\begin{cases}
  \temp \ln Z^{0,n}_{\temp}(0,0), & \temp \in (0,1], \\
  - A^{0,n}(0,0), & \temp = 0.
\end{cases}
\end{equation}
It follows from~(\ref{eq:limit-of-lnZ}) that $p_n(\temp)$ is continuous for $\temp \in [0,1]$.

\section{Concentration inequality for free energy}\label{sec:concentration}
The aim of this section is to prove a concentration inequality for the free energy~$p_n(\temp)$. In conjunction with the shape function convexity, it will help us to establish straightness estimates.  

\begin{theorem}
  \label{thm:concentration-of-free-energy}
  There are positive constants $c_0, c_1, c_2, c_3$ such that for all $n > c_0$ and all $u \in (c_3n^{1/2} \ln^{3/2}n,  n \ln n]$,
  \begin{equation*}
    \Pp \bigl\{ |p_n(\temp) - \alpha_{0; \temp} n | \le  u,\  \temp \in [0,1]\bigr\} \ge 1- c_1 \exp
    \left\{ -c_2 \frac{u^2}{n \ln^2n} \right\}.
  \end{equation*}
\end{theorem}
Similar inequalities for fixed $\temp$ were established in
~\cite{kickb:bakhtin2016} and~\cite{2016arXiv160704864B}. Theorem~\ref{thm:concentration-of-free-energy} is a nontrivial improvement of those bounds since it estimates the probability of the intersection of fixed $\temp$ events over all $\temp\in[0,1]$.

\subsection{A simpler concentration inequality}
\label{sec:simpler-concentration}

The first step in proving Theorem~\ref{thm:concentration-of-free-energy} is to obtain a concentration of~
$p_n(\temp)$ around its expectation.
\begin{lemma}
  \label{lem:free-energy-concentration-around-expectation}
  There are positive constants $b_0,b_1,b_2,b_3$ such that for all $n \ge b_0$, all $\temp \in
  [0,1]$ and all $u\in (b_3, n\ln n]$, 
  \begin{equation*}
    \Pp\Big\{ |p_n(\temp)- \E p_n(\temp)|\le  u \Big\}
    \ge 1-
    b_1 \exp \left\{ -b_2 \frac{u^2}{n\ln^2n} \right\}.
  \end{equation*}
\end{lemma}

In comparison with inequalities in~\cite{kickb:bakhtin2016} and~\cite{2016arXiv160704864B},
the important step here is choosing the constants~$b_i$'s  uniformly over all~$\temp \in [0,1]$. 
However, the event on the left-hand side is still defined for an arbitrary but fixed~$\temp \in [0,1]$. 
The proof of
this lemma is based on some auxiliary results that we prove first.

\smallskip

For $m<n$, we define 
\begin{equation*}
  \Sigma^{m,n}(\gamma) = 
  \bigg[ \sum_{j=m+1}^n
  \big(\gamma_{j} - \gamma_{j-1}  \big)^2     - \frac{(\gamma_n-\gamma_m)^2}{n-m}    \bigg]^{1/2}.
\end{equation*}
The function $\Sigma^{m,n}(\cdot)$ compares the action of a path $\gamma$ between time $m$ and $n$
to the action of the straight
line connecting $(m, \gamma_m)$ and $(n, \gamma_n)$.
It is also easy to check that $\Sigma^{m,n}(\cdot)$ is invariant under
space translations and shear transformations, namely, for any path $\gamma$, $\Sigma^{m,n}(\gamma) = \Sigma_2^{m,n}(\theta^{0,x}\gamma)=
\Sigma^{m,n}(L^v \gamma) $ for any $v,x\in \R$.

The next lemma summarizes various estimates which reflect the idea that  with high probability,   polymer measures assign small weights to
the path $\gamma$ that has large values of $\Sigma^{m,n}(\gamma)$.
To state the lemma, we need some more notations.

Let us define the set of paths
\begin{equation}\label{eq:def-of-E}
  E^{m,n}_s = \left\{ \gamma \ :\  \frac{1}{\sqrt{n}} \Sigma^{m,n}(\gamma) \in [s,s+1) \right\},
  \quad s \in \Z.
\end{equation}
For a Borel set $B \subset \R^{n-m-1}$, let us define
\begin{equation}
  \label{eq:partition-function-of-a-subset}
Z^{m,n}_{x,y; \temp}(B) = \int_{\R\times B\times \R}
e^{ - \be A^{m,n}(x_m, ..., x_n)} \, 
    \delta_x(dx_x)dx_m...dx_{n-1}\delta_y(dx_n).
\end{equation}
Let $\pi_{m,n}$ denote the restriction of a vector or sequence onto the time interval~$[m,n]_{\Z}$.
For a Borel set~$D \subset \R^{\infty} =  S^{-\infty,\infty}_{*,*}$, we
define 
\begin{equation*}
  \mu_{x,y; \temp}^{m,n}(D) = \mu_{x,y; \temp}^{m,n}(\pi_{m,n} D), \quad
  Z^{m,n}_{x,y; \temp}(D) = Z^{m,n}_{\temp}(x,y,D) = Z^{m,n}_{x,y; \temp} \mu_{x,y; \temp}^{m,n}(D).
\end{equation*}
\begin{lemma} 
  \label{lem:estimates-by-comparing-action}
  Let $n \ge 2$.
  There are constants $d_1> 0$, $R_1>0$ such that if $s, s' \ge R_1$, then the following statements hold:
  \begin{gather}
     \label{eq:partition-function-not-too-small}
      \Pp \Big\{ Z_{x, y; \temp}^{0, n}([0,1]^{n-1}) > 2^{-\be\cdot s n}, \ x, y \in [0,1], \ \temp \in (0,1]  \Big\} \ge 1-  e^{-d_1 s n}, \\
      \label{eq:partition-function-not-too-large}
   \Pp \Big\{    Z_{x,y; \temp}^{0,n} (  E_{s'}^{0,n} ) \le  2^{-\be \cdot 2s'n -1}, \  x,y\in [0,1],\ \temp \in (0,1] \Big\}    \ge
   1-  e^{-d_1s'n}, \\
\label{eq:partition-function-with-more-than-linear-action}
   \Pp \Big\{   Z_{x,y; \temp}^{0,n} \big( \bigcup_{s' \ge s} E_{s'}^{0,n} \big) \le  2^{-\be \cdot 2sn}, \  x,y\in [0,1]; \temp \in (0,1] \Big\}    \ge
   1-  2e^{-d_1sn}, \\
  \label{eq:polymer-measure-with-more-than-linear-action}
   \Pp \Big\{    \mu_{x,y; \temp}^{0,n} \big( \bigcup_{s' \ge s} E_{s'}^{0,n} \big) \le    2^{-\be \cdot sn}, \  x,y\in [0,1],\temp \in (0,1] \big\}    \ge
   1-  3e^{-d_1s n}, \\
\label{eq:large-deviation-for-Z-bar}
\begin{multlined}
  \Pp \Big\{ \mu_{x,y; \temp}^{0,n} \big\{ \gamma: \frac{1}{n}\max\limits_{0\le j \le n}|\gamma_j|
  \ge s \Big\} \le 2^{-\be \cdot sn} ,  \ x,y\in [0,1],\temp \in (0,1] 
  \Big\}  \\
\ge 1-  3e^{-d_1sn}.
\end{multlined}
\end{gather}
\end{lemma}

\begin{proof}
  It suffices to show~(\ref{eq:partition-function-not-too-small}) and~(\ref{eq:partition-function-not-too-large}).
  Then~(\ref{eq:partition-function-with-more-than-linear-action}) will follow
  from~(\ref{eq:partition-function-not-too-large}) by summing over integer $s \ge s'$, 
  and~(\ref{eq:polymer-measure-with-more-than-linear-action}) from~(\ref{eq:partition-function-not-too-small})
  and~(\ref{eq:partition-function-with-more-than-linear-action}).  Finally,
  the convexity of~$z \mapsto z^2$ and Jensen's inequality imply that for all $\gamma \in
  S_{x,y}^{0,n}$ and all  $x,y\in[0,1]$,
  \begin{align*}
    [\Sigma^{0,n}(\gamma) \big]^2
    &\ge  \sum_{j=1}^n |\gamma_j - \gamma_{j-1}|^2
    -\frac{1}{n} \ge \frac{1}{n} \biggl( \sum_{j=1}^n |\gamma_{j}-\gamma_{j-1}| \biggr)^2
      -\frac{1}{n} \\
&\ge \frac{1}{n} \Big[  2\bigl( \max_{1\le j \le n-1}|\gamma_j| - 1 \bigr)_+ \Big]^{2} -\frac{1}{n}.
  \end{align*}
  Therefore, when~$s$ is large,  $\max\limits_{1\le j \le n-1}|\gamma_j| \ge sn$ implies
  $\Sigma^{0,n}(\gamma) \ge s \sqrt{n}$, so~(\ref{eq:large-deviation-for-Z-bar}) holds.

By definition~(\ref{eq:partition-function-of-a-subset}), we have
\begin{equation*}
\begin{split}
  Z_{x,y; \temp}^{0, n}([0,1]^{n-1})  \ge  e^{-\be[n/2+
    F_{\omega}^{*}(0,...,0) ]} ,
  \quad x, y \in [0,1],
\end{split}
\end{equation*}
where  $F_{\omega}^{*}(i_1,...,i_n) = \sum_{j=1}^n F_{\omega}^{*}(j,i_j)$
(see~(\ref{eq:def-F-star}) for the definition of $F^{*}_{\omega}$).
So, for all $x,y \in [0,1]$, $\temp \in (0,1]$,
\begin{align}
    \label{eq:first-inclusion} 
  \bigl\{ \omega :Z_{x,y; \temp}^{0,n} ([0,1]^{n-1}) <  2^{-\be\cdot s n} \bigr\}
\subset   \bigl\{\omega:  n( s \ln 2 - 1/2) < F_{\omega}^{*}(0,...,0)
  \bigr\}. 
\end{align}
By Markov inequality,  we have
\begin{equation}\label{eq:cramer-for-potential}
  \Pp \bigg\{ F_{\omega}^{*}(0,...,0) > r   \bigg\} \le
  e^{-\eta r}\E e^{ \sum_{j=1}^n \eta F^{*}_{\omega}(j,0)} \le e^{-\eta r} \bigl( \E e^{\eta F^{*}_{\omega}(0,0)} \bigr)^n.
\end{equation}
Combining~(\ref{eq:first-inclusion}) and~(\ref{eq:cramer-for-potential}), we
obtain~(\ref{eq:partition-function-not-too-small}):  for sufficiently large~$s$,
\begin{align*}
  &\Pp \left\{ Z_{x,y; \temp}^{0,n}([0,1]^{n-1}) > 2^{-\be\cdot s n}, \ x, y \in [0,1]; \temp \in (0,1]  \right\} \\
  \ge &1  - \Pp \bigl\{\omega:  n( s \ln 2 -1/2) < F_{\omega}^{*}(0,...,0)  \bigr\} \\
  \ge &1 - e^{-n\cdot \eta \big( s \ln 2 -1/2)} \bigl( \E e^{\eta F^{*}_{\omega}(0,0)} \bigr)^n.
\end{align*}

Next we turn to~(\ref{eq:partition-function-not-too-large}).
In proving this, we will write $s$ instead of $s'$.
Let us define
\begin{equation*}
S^n_s = \{ (i_1, ..., i_{n-1}): \exists \gamma \in \tilde{E}^{0,n}_s, x, y \in [0,1]
\text{ s.t. } [\gamma_j] = i_j, 1 \le j \le n-1 \},
\end{equation*}
where~${\tilde{E}^{0,n}_s  = E^{0,n}_s \cap \big( \bigcup_{x,y \in [0,1]} S_{x,y}^{0,n} \big)}$.
Then we have
\begin{equation}\label{eq:smallness-of-Z-E}
Z_{x,y; \temp}^{0,n} (E^{0,n}_s) 
  \le  |S^n_s|  e^{\be  \bigl(-\frac{1}{2}s^2n  + F_{\omega,n,s}^{*}\bigr)}, \quad x,y \in
  [0,1], \temp \in (0,1],
\end{equation}
where~$F_{\omega,n,s}^{*} = \max\{ F_{\omega}^{*}(i_1,..,i_{n-1},0) : (i_1,...,i_{n-1})\in S^n_s\}$.

We need to estimate the size of $S^n_s$. For $1 \le j  \le n$, let us define $k_j = \gamma_j-\gamma_{j-1}$ and~${\tilde{k}_j = 
[\gamma_j]-[\gamma_{j-1}]} $. 
Clearly, $|k_j - \tilde{k}_{j}| \le 2$.
If $\gamma \in \tilde{E}_s^{0,n}$, then the  Cauchy--Schwarz inequality  implies
\begin{equation*}
 \sum_{j=1}^n|k_j|  \le \sqrt{n} \sqrt{\sum_{j=1}^n k_j^2} \le \sqrt{(s+1)n^2 + n}.
\end{equation*}
Comparing $\sum_{j=1}^nk_j^2$ and $\sum_{j=1}^n\tilde{k}_j^2$, we obtain
\begin{equation*}
  \biggl|   \sum_{j=1}^nk_j^2 - \sum_{j=1}^n\tilde{k}_j^2  \biggr|
  \le \sum_{j=1}^n |k_j - \tilde{k}_j| |k_j + \tilde{k}_j |
  \le 2 \sum_{j=1}^n (2|k_j| + 2)
  \le 8sn.
\end{equation*}
Therefore, $\gamma \in \tilde{E}_{s}^{0,n}$ implies that 
\begin{equation}\label{eq:condition-for-k-tilde}
  \sum_{j=1}^n \tilde{k}_j^2  \le (s+1)^2 n + 8sn =: [r_s(n)]^2.
\end{equation}
The size of $S_s^n$ is bounded by the number of $n$-vectors~$(\tilde{k}_0, ..., \tilde{k}_{n-1})$
satisfying~(\ref{eq:condition-for-k-tilde}), which is then bounded by the volume of $n$-dimensional
ball of radius $r_s(n) + \frac{\sqrt{n}}{2}$. (To obtain this estimate, we consider unit cubes centered at integer
points, with half diagonal lengths $\frac{\sqrt{n}}{2}$.)
Hence, when $s$ is large,
\begin{equation}\label{eq:bound-for-S}
\begin{split}
  |S^n_s| \le&\frac{\pi^{n/2}}{\Gamma(n/2+1)} \big(  r(n) + \frac{\sqrt{n}}{2} \big)^n\\
  \le & \frac{\pi^{n/2}}{\Gamma(n/2+1)} \cdot ( K_1 s \sqrt{n})^{n}    \le  e^{(\ln s + K_2) n},
\end{split}
\end{equation}
where $K_1, K_2$ are constants and we used $\ln \Gamma(z) = z\ln z -z + \mathrm{O}(\ln z)$, $z\to\infty$.

Combining~(\ref{eq:smallness-of-Z-E}) and ~(\ref{eq:bound-for-S}) , we see that for $x,y \in [0,1]$, $\temp \in (0,1]$
and large $s$,
\begin{equation}
\begin{split}
 &\bigl\{ \omega : Z_{x,y; \temp}^{0,n} (E^{0,n}_s) > 2^{-\be \cdot 2s n - 1} \bigr\} \\
\subset&  \bigl\{\omega :   -\frac{1}{2}s^2n   + F_{\omega,n,s}^{*}+ \temp \ln|S^n_s| > - 2sn- \temp \bigr\} \\
  \subset& \bigl\{ \omega : F^{*}_{\omega,n,s} >  \frac{1}{2}s^2n  -2sn - \temp(1+\ln |S^n_s|) \bigr\}\\
  \subset& \bigl\{ \omega : F^{*}_{\omega,n,s} >  sn\bigr\}.
\end{split}
\label{eq:second-inclusion}
\end{equation}
Since the distribution of $F_{\omega}^{*}(i_0,...,i_{n-1})$ does not depend on the choice of
the vector $(i_0,...,i_{n-1})$, we obtain that for any $r>0$, 
\begin{equation}\label{eq:2}
  \Pp\{F^{*}_{\omega,n,s} > r \} \le
  |S^n_s| \Pp \{F^{*}_{\omega}(0,...,0) > r  \}.
\end{equation}
Combining~(\ref{eq:cramer-for-potential}), (\ref{eq:bound-for-S}),  (\ref{eq:smallness-of-Z-E}),
(\ref{eq:second-inclusion}), and~(\ref{eq:2}), we see that 
\begin{align*}
 &\Pp \bigg\{    Z_{x, y; \temp}^{0, n} ( E^{0,n}_s ) \le 2^{-\be \cdot 2sn - 1}  , \ x, y \in [0,1]; 
 \ \temp
   \in (0,1] \bigg\}     \\
\ge  &1-   |S^n_s|\Pp \{   F^{*}_{\omega,n,s} >  sn  \} \\
 \ge & 1- e^{(\ln s + K_2) n} e^{-\eta s n} \bigl( \E e^{\eta F^{*}_{\omega}(0,0)} \bigr)^n 
\end{align*}
Choosing $s$ large enough concludes the proof of~(\ref{eq:partition-function-not-too-large}).
\end{proof}

Let $E^{m,n}_{ \le R_1} = \bigcup_{s \le R_1}E^{m,n}_s$.
The following lemma states that $Z^{0,n}_{x,y; \temp}(E^{0,n}_{\le R_1})$ cannot be large.
\begin{lemma}
\label{lem:partition-function-inside-the-box-not-big}
There is some constant $d$ such that for sufficiently large $t$, 
\begin{equation*}
\Pp \Big\{ Z_{x,y; \temp}^{0, n}(E^{0,n}_{\le R_1} ) \le e^{\be tn-1}, \ x, y \in [0,1];\ \temp \in (0,1] \Big\} \ge 1 - e^{-d tn}.
\end{equation*}
\end{lemma}

\begin{proof}
  We will continue using the notations from the proof of Lemma~\ref{lem:estimates-by-comparing-action}.
  Let us define~$S_{\le R_1}^n = \bigcup_{s \le R_1}S^n_s$.
  Similarly to~(\ref{eq:bound-for-S}) and~(\ref{eq:smallness-of-Z-E}), we have
\begin{equation}\label{eq:upper-bound-for-S-less-than-R1}
|S^n_{\le R_1} | \le \frac{\pi^{n/2}}{\Gamma(n/2+1)} \bigl( r_{R_1}(n) + \sqrt{n}/2 \bigr)^n  \le  e^{K_1 n}
\end{equation}
for some constant $K_1$, and
\begin{equation}
Z_{x,y; \temp}^{0, n} (E^{0,n}_{\le R_1}) 
  \le  |S^{n}_{\le R_{1}}|  e^{ \be  F_{\omega,n, \le R_1}^{*}}, \quad x, y \in [0,1];\ \temp
  \in (0,1],
\end{equation}
where $ F_{\omega,n,\le R_1}^{*} = \max\{ F_{\omega}^{*}(i_1,...,i_{n-1}) : (i_1,...,i_{n-1})\in S^n_{\le R_1}\}$.
Therefore, for~$x,y \in [0,1]$, $\temp \in (0,1]$ and sufficiently large $t$,
\begin{equation*}
\begin{split}
 \bigl\{ \omega : Z_{x,y; \temp}^{0,n} (E^{0,n}_{\le R_1}) > e^{\be tn-1} \bigr\} 
\subset&  \bigl\{\omega :    F_{\omega,n,\le R_1}^{*}+ \temp \ln|S^n_{\le R_1}| >tn -
\temp  \bigr\} \\
\subset&  \bigl\{\omega :    F_{\omega,n,\le R_1}^{*} >tn - \temp ( \ln |S^n_{\le R_1}| + 1)  \bigr\} \\
  \subset& \bigl\{ \omega : F^{*}_{\omega,n,\le R_1} >  tn/2\bigr\}.
\end{split}
\end{equation*}
Combining this with~(\ref{eq:cramer-for-potential}) and~(\ref{eq:upper-bound-for-S-less-than-R1}), we obtain
\begin{align*}
&  \Pp \Big\{ Z_{x,y; \temp}^{0, n}(E^{0,n}_{\le R_1} ) \le e^{\be tn-1}, \ x, y \in [0,1];\ \temp \in (0,1] \Big\}\\
  \ge& 1 - \Pp \bigl\{ \omega : F^{*}_{\omega,n,\le R_1} >  tn/2\bigr\}\\
  \ge & 1- |S^n_{\le R_1} | \Pp\bigl\{ \omega : F^{*}_{\omega}(0,...,0) >  tn/2\bigr\} \\
  \ge & 1-  e^{-(\eta t/2 - K_1) n} \bigl( \E e^{\eta F^{*}_{\omega}(0,0)} \bigr)^n.
\end{align*}
Choosing $t$ large enough concludes the proof.
\end{proof}

Combining~(\ref{eq:partition-function-with-more-than-linear-action}) with $s = R_1$ and
Lemma~\ref{lem:partition-function-inside-the-box-not-big}, we obtain the following lemma.
\begin{lemma}
  \label{lem:Z-n-less-than-exponential}
  There are constants $d_2, R_2 > 0$ such that for all $t \ge R_{2}$, 
\begin{equation*}
\Pp \Big\{ Z^{0,n}_{x,y; \temp} \le e^{\be tn}, \ x,y \in [0,1];\ \temp \in (0,1] \Big\} \ge 1 - e^{-d_2 tn}.
\end{equation*}
\end{lemma}
Also, as a consequence of~(\ref{eq:upper-bound-for-S-less-than-R1}), we have the following upper
bound for the Lebesgue measure of~$E^{0,n}_{\le R_1}$. 
\begin{lemma}
\label{lem:exponential-bound-for-E}
There is a constant $d_3 > 0$ such that $|E^{0,n}_{\le R_1}| \le e^{d_3 n}$.
\end{lemma}

Using Lemma~\ref{lem:Z-n-less-than-exponential} and~(\ref{eq:partition-function-not-too-small}) of
Lemma~\ref{lem:estimates-by-comparing-action}, we have estimates on all moments of the logarithm of partition functions.
\begin{lemma}
  \label{lem:second-moment-growth-of-partition-function}
  There are constants $M(p)$, $p \in \N$, such that for all $\temp \in (0,1]$ and any Borel set $B$
  satisfying $[0,1]^{n-1} \subset B \subset \R^{n-1}$, 
  \begin{equation*}
    \E |\temp \ln Z_{0,0; \temp}^{0,n}(B)|^{p} \le M(p) n^p.
  \end{equation*}
\end{lemma}

Let us denote  $Z^{0,n}_{0,0; \temp}$ by $Z^n_{\temp}$.  
\begin{lemma} \label{lem:truncated-log-partition-function-expectation}
There is a constant $D_1 > 0$ such that
\[
 0 \le \temp\Big(  \E \ln Z_{\temp}^n -  \E \ln Z_{\temp}^{n}(E^{0,n}_{\le R_1})  \Big) \le D_1,
 \quad n\in\N,\ \temp \in (0,1].
\]

\end{lemma}

\begin{proof}
The first inequality is obvious since $Z_{\temp}^n(E^{0,n}_{\le R_1}) \le Z_{\temp}^n$. 
Let
$$\Lambda = \{ Z^n_{\temp}(E^{0,n}_{\le R_1}) /Z_{\temp}^n \le 1 - 2^{-\be R_1n} \}.$$
By~(\ref{eq:polymer-measure-with-more-than-linear-action}) of Lemma~\ref{lem:estimates-by-comparing-action},
$\Pp(\Lambda) \le 3e^{-d_1 R_1n}$.
By Lemma~\ref{lem:second-moment-growth-of-partition-function}, we have
\begin{equation*}
\E |\temp\ln Z_{\temp}^n (E^{0,n}_{ \le R_1})|^2 \le M(2)n^2, \quad \E |\temp\ln Z_{\temp}^n|^2 \le M(2)n^2.
\end{equation*}
The lemma then follows from
\begin{align*}
&   \temp \Big( \E\ln Z_{\temp}^n - \E\ln Z_{\temp}^n (E^{0,n}_{ \le R_1}) \Big)\\
\le  &- \temp  \E \ln \big(Z_{\temp}^n (E^{0,n}_{\le R_1})/Z_{\temp}^n  \big) \ONE_{\Lambda^c}  +
    \temp \E(|\ln Z_{\temp}^n| + |\ln Z^n_{\temp}(E^{0,n}_{\le R_1})|)\ONE_{\Lambda} \\
\le  & -\temp \ln(1-2^{-\be R_1n}) + \temp\sqrt{2(\E \ln^2 Z_{\temp}^n + \E \ln^2 Z_{\temp}^n (E^{0,n}_{\le R_1}))   } \sqrt{\Pp (\Lambda)} \\
\le  & |\ln (1-2^{-R_1})|+ \sqrt{4M(2)n^2 \cdot 3e^{-d_1R_1n}} .
\end{align*}
\end{proof}

Let us define 
\begin{equation*}
\tilde{p}_n(\temp) =
\begin{cases}
  \temp \ln  Z_{\temp}^{n}(E^{0,n}_{\le R_1}),& \temp \in (0,1], \\
  -\min \{A^{0,n}(\gamma): \  \gamma \in S^{0,n}_{0,0} \cap E^{0,n}_{\le R_1}  \},  & \temp = 0.
\end{cases}
\end{equation*}
Clearly, $\tilde{p}_n(\cdot)$ is continuous on~$[0,1]$.  We recall that $p_n(\cdot)$ defined
in~(\ref{eq:def-of-finite-free-energy}) is also continuous on~$[0,1]$.  Since
Lemma~\ref{lem:second-moment-growth-of-partition-function} implies uniform integrability of~$\big(
p_n(\temp) \big)_{\temp \in (0,1]}$ and~$\big( \tilde{p}_n (\temp)\big)_{\temp \in (0,1]}$, we
immediately obtain that both $\E p_n(\temp)$ and~$\E \tilde{p}_n(\temp)$ are continuous for $\temp \in [0,1]$.
The next lemma estimates how well~$\tilde{p}_n(\temp)$ approximates~$p_n(\temp)$.
\begin{lemma}
  \label{lem:difference-between-p-and-p-tilde}
  If $n$ is sufficiently large, then for all $\temp \in [0,1]$, 
\begin{equation}\label{eq:p-tilde-close-to-p-in-probability}
  \Pp \big\{ |p_n(\temp) - \tilde{p}_n(\temp)| \le 1, \ \temp \in [0,1]
  \big\} \ge 1 -  3 e^{-d_1R_1n}
\end{equation}
and 
\begin{equation}\label{eq:p-tilde-close-to-p-in-expectation}
|\E p_n(\temp) - \E \tilde{p}_n(\temp)| \le D_1, \quad\temp \in [0,1].
\end{equation}
\end{lemma}

\begin{proof}
Due to 
\eqref{eq:polymer-measure-with-more-than-linear-action}, we have
\begin{align*}
&\quad  \Pp \big\{  |p_n(\temp) - \tilde{p}_n(\temp)| \le |\ln (1-2^{-\be \cdot R_1n}) |, \  \temp \in (0,1]
  \big\} \\
& \ge    \Pp \Big\{    \mu_{0,0; \temp}^{0,n} \big( \bigcup_{s' \ge R_1} E_{s'}^{0,n} \big)  \le   2^{-\be \cdot R_1n}, \  \temp \in (0,1] \Big\}    
  \ge 1 -  3 e^{-d_1R_1n}.
 \end{align*}
Then~(\ref{eq:p-tilde-close-to-p-in-probability}) follows from this and the continuity of $p_n$ and
$\tilde{p}_n$ in $\temp$.
The second inequality~(\ref{eq:p-tilde-close-to-p-in-expectation}) follows from
Lemma~\ref{lem:truncated-log-partition-function-expectation} and
the continuity of $\E p_n$ and
$\E\tilde{p}_n$ in $\temp$.
\end{proof}

To obtain a concentration inequality for $\tilde{p}_n(\temp)$, we need Azuma's inequality:
\begin{lemma}
  \label{lem:azuma}
  Let $(M_k)_{0\le k\le N}$ be a martingale with respect to a
  filtration $(\Fc_k)_{0\le k\le N}$.
  Assume there is a constant $c$ such that~$|M_k - M_{k-1}| \le c$, $1 \le k \le N$.
Then
\begin{equation*}
  \Pp \left\{ |M_N - M_0| \ge x \right\}  \le 2 \exp \left( \frac{-x^2}{ 2Nc^2} \right).
\end{equation*}
\end{lemma}

To apply Azuma's inequality, we need to introduce an appropriate martingale with bounded increments.
The function $\tilde{p}_n(\temp)$ depends only on the potential process on $B = \{1,\dotsc,n\}\times
[-R_1n, R_1n]$ since $\pi^{1,n} E^n_{\le R_1} \subset [-R_1n, R_1n]^n$, 
so we need an additional truncation of the potential on~$B$.
Moreover, the truncation should be independent of~$\temp$.

Let $b>4/\eta$, where $\eta$ is taken from the condition \ref{item:exponential-moment-for-maximum}.
For~$1\le k \le n$ and~$x \in [-R_1n, R_1n]$, we define (suppressing the dependence on $n$ for brevity)
\[\xi_k = \max\{F^{*}_k(j): j=-R_1n, -R_1n+1, ..., R_1n-1 \},\quad k=0,\ldots,n,\] 
\begin{equation*}
  \bar{F}_k(x)
  = \begin{cases}
0, & \xi_k \ge b\ln n,\\
F_k(x), & \text{otherwise},
\end{cases}
\end{equation*}
and setting~$x_0=x_n=0$,
\begin{equation*}
  \tilde{p}_n(\temp, \bar{F}) =
  \begin{cases}
\temp \ln     \int_{\pi_{0,n} E^{0,n}_{\le R_1}} \prod\limits_{j=1}^n g_{\temp}(x_j-x_{j-1})
    e^{-\be \cdot \bar{F}_j(x_j)} \delta_0(dx_0)dx_1\ldots dx_{n-1}\delta_0(dx_n), & \temp \in (0,1],\\
    - \min\limits_{ \substack {(x_0, x_1, ..., x_{n-1}, x_n) \in \pi_{0,n} E^n_{\le R_1} \\ x_0 =
        x_n = 0} }  \sum\limits_{j = 1}^n \big[  \frac{1}{2}(x_j - x_{j-1})^2 + \bar{F}_j(x_j)
    \big], & \temp  =0.
\end{cases}
\end{equation*}

\begin{lemma}
  \label{lem:partition_function_of_omega_bar_and_omega}
 For sufficiently large $n\in\N$, the following holds true:
\begin{gather}
 \label{eq:item:3}
\E \exp \Bigl(  \frac{\eta}{2} \xi_k \ONE_{\{\xi_k \ge b\ln n\}}   \Bigr) \le 2,\\
\label{eq:item:4} \E \xi_k \le b\ln  n + 4/\eta,\\
\label{eq:item:8}  \Pp \{  |\tilde{p}_n(\temp)   - \tilde{p}_n(\temp, \bar{F})) | \le x, \ \temp \in [0,1]\}
    \ge 1 - 2e^{-\eta x/2},\quad x > 0,\\
\label{eq:item:9}
| \E \tilde{p}_n(\temp) - \E \tilde{p}_n(\temp,\bar{F}) | \le  4/\eta, \quad \temp \in [0,1].
\end{gather}
\end{lemma}

\begin{proof}  Since $\xi_k$ is the maximum of $2R_1n$ random variables with the same distribution, we have
  \begin{align}\notag
 \E \exp \Bigl( \frac{\eta}{2}\xi_k  \ONE_{\{\xi_k > b\ln n\}}\Bigr) 
  &\le  1 + \E e^{\frac{\eta}{2}\xi_k} \ONE_{\{\xi_k > b\ln n\}} \le 1 + \E \sum_{j=-R_1n}^{R_1n-1} e^{\frac{\eta}{2} F^{*}_k(j)} 
\ONE_{\{F^{*}_k(j) > b\ln n\}} \\
\notag
  &\le 1 + 2R_1n  \E e^{ \frac{\eta}{2} F^{*}_k(0)} \ONE_{\{ F^{*}_k(0) > b\ln n\}} \le 1 + 2R_1n \frac{\E e^{\eta F^{*}_k(0)}}{e^{ \frac{b\eta}{2} \ln n}} \\
  & \le 1 + \frac{c}{n^{\frac{b\eta}{2}-1}},
\label{eq:exp-moment-above-bln}
  \end{align}
  where $c = 2R_1 \E e^{\eta F^{*}_k(0)}$ is a constant. Now \eqref{eq:item:3} follows from $b>4/\eta$.

  If $x > b\ln n$, then by Markov inequality and~\eqref{eq:item:3}, we have
\[
  \Pp \{\xi_k \ge x \} \le \Pp \{\xi_k \ONE_{\{ \xi_k \ge b\ln n\}} \ge x\} \le e^{-\eta x/2} 
\E \exp \Bigl( \frac{\eta}{2} \xi_k \ONE_{\{ \xi_k \ge b \ln n\}} \Bigr) \le 2e^{-\eta x/2}
\] for sufficiently large $n$.
This implies~\eqref{eq:item:4}:
\[
  \E \xi_k \le b \ln n + \E \xi_k \ONE_{\{\xi_k \ge b\ln n\}} 
  \le b \ln n  + \int_{b\ln n}^{\infty} \Pp \{\xi_k  \ge x\} \, dx \le b\ln
  n + \frac{4}{\eta}.
\]
It follows from the definition  of $\tilde{p}_n(\temp, \bar{F})$ that for all $\temp \in [0,1]$,
\begin{equation*}
    | \tilde{p}_n(\temp)  - \tilde{p}_n(\temp, \bar{F}) |\le 
    \sum_{k=1}^n \xi_k\ONE_{\{\xi_k > b\ln n\}}.
  \end{equation*}
  By Markov inequality, the \iid property of $(\xi_k)$ and~\eqref{eq:exp-moment-above-bln}, we have
  \begin{align*}
\Pp \left\{     | \tilde{p}_n(\temp)  - \tilde{p}_n(\temp, \bar{F}) | \le x,  \ \temp
   \in [0,1]\right\}    
  & \ge 1-  \Pp \Bigl\{ \frac{\eta}{2}\sum_{k=1}^n \xi_k \ONE_{\{\xi_k > b\ln n\}} > \frac{\eta x}{2} \Bigr\} \\
  & \ge 1-  e^{-\eta x/2} \E \exp \Bigl(\frac{\eta}{2} \sum_{k=1}^n \xi_k  \ONE_{\{\xi_k> b\ln n\}} \Bigr)\\
  & = 1- e^{-\eta x/2} \Big(\E \exp \big(  \frac{\eta}{2}\xi_0  \ONE_{\{\xi_0> b\ln n\}}\big) \Big)^{n} \\
    & \ge 1-  e^{-\eta x/2} (1+ c/n^{\eta b/2-1})^{n}.
\end{align*}
Since $b>4/\eta$,~\eqref{eq:item:8} follows. It immediately implies

\begin{equation*}
\begin{split}
  | \E \tilde{p}_n(\temp)  - \E\tilde{p}_n(\temp, \bar{F}) |
  &\le \E     | \tilde{p}_n(\temp)  - \tilde{p}_n(\temp, \bar{F}) | \\
  &= \int_0^{\infty} \Pp \{     | \tilde{p}_n(\temp)  - \tilde{p}_n(\temp, \bar{F}) | > x \} \, dx
  \le 4/\eta,  
\end{split}
\end{equation*}
so \eqref{eq:item:9} is also proved.
\end{proof}  

\begin{lemma}
\label{lem:martingale-concentration-for-ZF-bar}
For all $n\in\N$,  $x>0$ and all $\temp \in [0,1]$,
  \begin{equation*}
    \Pp \left\{    |\tilde{p}_n(\temp, \bar{F})  - \E \tilde{p}_n(\temp, \bar{F})|
      > x \right\} \le 2\exp \left\{
      -\frac{x^2}{8nb^2\ln^2 n} \right\}.
  \end{equation*}
\end{lemma}

\begin{proof}
    Let us introduce the following martingale $(M_k,\Fc_k)_{0\le k \le n}$:
\begin{equation*}
  M_k = \E ( \tilde{p}_n(\temp, \bar{F}) \,|\, \Fc_k ), \quad 0\le k \le n,
\end{equation*}
where 
\begin{equation*}
  \Fc_{0}= \{\emptyset,\Omega \}, \quad
  \Fc_k = \sigma \big( F_{i,\omega}(x) : 1 \le i \le k \big), \quad  k=1,\dotsc,n.
\end{equation*}
If we can show that $|M_{k} - M_{k-1}| \le 2b\ln n$, $ 1 \le k \le n$, 
then the conclusion of the lemma follows immediately from Azuma's inequality (Lemma~\ref{lem:azuma}).

For a process $\bar G$, an  independent distributional copy of $\bar F$, let us define
\begin{multline*}
  \tilde{Z}_{\temp}^n ([\bar{F},\bar{G}]_k) = \int_{|x_i|\le R_1n}
\prod_{i=1}^k g_{\temp}(x_{i}-x_{i-1}) e^{- \be \bar{F}_i(x_i)} \\ \cdot \prod_{i=k+1}^n g_{\temp}(x_i-x_{i-1})
e^{-\be \bar{G}_i(x_i)} \delta_0(d x_0) dx_1\dotsm dx_{n-1} \delta_0(d x_n).
\end{multline*}
Denoting by $P_k$ the distribution of $\bar{F}_{k}(\cdot)$, we obtain for $\temp \in (0,1]$,
\begin{equation*}
\begin{split}
&|M_{k} - M_{k-1}| 
\\=& \temp \Bigg| \int \ln \tilde{Z}^n_{\temp}([\bar{F},\bar{G}]_k) \prod_{i=k+1}^n P_i\big( d
\bar{G}_i \big) -
\int \ln \tilde{Z}^n_{\temp}([\bar{F},\bar{G}]_{k-1}) \prod_{i=k}^n P_i\big( d \bar{G}_i \big) \Bigg|
\\
\le& \temp\int \Big|  \ln \tilde{Z}^n_{\temp}([\bar{F},\bar{G}]_k) - \ln \tilde{Z}^n_{\temp}([\bar{F},\bar{G}]_{k-1}) \Big|
\prod_{i=k}^n P_i\big( d \bar{G}_i \big) \\
\le& \int \Big( \sup_{|x| \le R_1 n} |\bar{F}_k(x)| + \sup_{|x| \le R_1n} | \bar{G}_k(x)| \Big)
\prod_{i=k}^n P_i\big( d \bar{G}_i \big) \le 2b\ln n,
\\
\end{split}
\end{equation*}
since $|\bar{F}_k(x)|$ and $|\bar{G}_k(x)|$ are bounded by  $b \ln n$.
By taking $\temp \downarrow 0$ in the above inequality (or using that resampling the
potential field $\big( F_i(\cdot) \big)$ at any given $i$ will change the optimal action by at most $2b
\ln n$), we can see that $|M_k - M_{k-1}| \le 2b \ln n$ also holds when $\temp = 0$.
This completes the proof.
\end{proof}

We note that in lemma~\ref{lem:martingale-concentration-for-ZF-bar}, we estimate the probability of an event defined for a fixed $\temp$, since 
the Azuma inequality applies to a fixed martingale and cannot be immediately used for uniform
concentration of a family of martingales parametrized by $\temp$.

\begin{proof}[Proof of Lemma \ref{lem:free-energy-concentration-around-expectation}]
  Suppose $u \in \big( 3(D_1+4/\eta + 3), n \ln n \big] $.   Then 
\begin{align*}
     &\quad  \Pp \big\{ |p_n(\temp) - \E p_n(\temp)| > u \big\} \\
&  \le \Pp \bigl\{ |p_n(\temp) - \tilde{p}_n(\temp)| > 1 \bigr\} 
 +  \Pp \bigl\{ | \tilde{p}_n(\temp)-   \tilde{p}_n(\temp, \bar{F}) | > \frac{u}{3} \bigr\} \\
& +  \Pp \bigl\{ | \tilde{p}_n(\temp, \bar{F})  - \E  \tilde{p}_n(\temp, \bar{F}) |  >
  \frac{u}{3}\bigr\} 
 +  \Pp \bigl\{  |\E  \tilde{p}_n(\temp, \bar{F})  - \E \tilde{p}_n(\temp)|  > 4/\eta +1 \bigr\} \\
&  + \Pp \bigl\{ |\E \tilde{p}_n(\temp) - \E p_n(\temp)| > D_1 +1  \bigr\}.
\end{align*}
By~\eqref{eq:item:9} and~(\ref{eq:p-tilde-close-to-p-in-expectation}), the last two terms equal 0.
The first three terms can be bounded by using~(\ref{eq:p-tilde-close-to-p-in-probability}), 
\eqref{eq:item:8} and Lemma~\ref{lem:martingale-concentration-for-ZF-bar}, respectively.
Combining all these estimates together, we obtain
\begin{equation*}
    \Pp \big\{ |p_n(\temp) - \E p_n(\temp)| > u \big\} \\  
  < 3e^{-d_1R_1n} + 2e^{-\frac{\eta u}{6}} + 2e^{ -
    \frac{u^2}{72b^2n \ln^2 n}}
  \le b_1 e^{-b_2 \frac{u^2}{n\ln^2 n}},
\end{equation*}
for some constants $b_1,b_2 > 0$, where in the last inequality we use $u \le n\ln n$.
\end{proof}

We also have obtained a similar concentration inequality for $\tilde{p}_n(\temp)$ which will be used
in the next section.
\begin{lemma}
\label{lem:concentration-for-p-tilde}
Let $b_i$'s be the constants in Lemma~\ref{lem:free-energy-concentration-around-expectation}.  Then
 for all $n \ge b_0$, all~$\temp \in  [0,1]$ and all $u\in (b_3, n\ln n]$, 
  \begin{equation*}
    \Pp\Big\{ |\tilde{p}_n(\temp)- \E \tilde{p}_n(\temp)|\le  u \Big\}
    \ge 1-
    b_1 \exp \left\{ -b_2 \frac{u^2}{n\ln^2n} \right\}.
  \end{equation*}
\end{lemma}

\subsection{Uniform continuity of the shape function in viscosity}
To go from Lemma~\ref{lem:free-energy-concentration-around-expectation} to
Theorem~\ref{thm:concentration-of-free-energy}, 
we have to estimate the difference of~$\E p_n(\temp)$ and~$\alpha_{0; \temp} n$,  and to move $\temp
\in [0,1]$ inside the events of interest.
The key point is to establish the continuity of $\alpha_{0; \temp}$ for $\temp \in [0,1]$.
\begin{lemma}
  \label{lem:continuity-of-alpha}
  
\begin{enumerate}
\item There is a constant $b_4$ such that for sufficiently large $n$, 
\begin{equation}\label{eq:error-estimate}
| \E p_n(\temp)  -  \alpha_{0; \temp} n | \le b_4 n^{1/2} \ln^2 n, \quad \temp \in [0,1].
\end{equation}
\item $\alpha_{0; \temp}$ is  continuous $\temp \in [0,1]$.
\end{enumerate}

\end{lemma}

Let us derive Theorem~\ref{thm:concentration-of-free-energy} from~\ref{lem:continuity-of-alpha} and
the results from section~\ref{sec:simpler-concentration} first.\\
\begin{proof}[Proof of Theorem~\ref{thm:concentration-of-free-energy}]
  Let us define 
\begin{equation*}
q_n(\temp) = \tilde{p}_n(\temp) - \temp \ln |E^{0,n}_{\le R_1}|, \quad \temp \in [0,1],
\end{equation*}
where~$|\cdot|$ denotes the Lebesgue measure of a set.
When $\temp > 0$, we have
\begin{equation*}
q_n(\temp) = \ln \Big(  \int_{E^{0,n}_{\le R_1}} \frac{1}{|E^{0,n}_{\le R_1}|} e^{-\be
  A^{0,n}(\gamma) } \, d \gamma \Big)^{\kappa}.
\end{equation*}
Therefore, by Lyapunov's inequality, $q_n(\temp)$ is decreasing in $\temp$.
Then by Lemma~\ref{lem:concentration-for-p-tilde}, for all $n \ge b_0$, all~$\temp \in [0,1]$ and~$x \in [b_3, n \ln n]$, 
\begin{equation}
\label{eq:concentratin-for-q}
\Pp\Big\{ |q_n(\temp)- \E q_n(\temp)|\le  x \Big\}
    \ge 1-
    b_1 \exp \left\{ -b_2 \frac{x^2}{n\ln^2n} \right\}.
\end{equation}

For fixed $n$, since $q_n(\cdot)$ is a continuous decreasing function, we can find~$M$ and $0 = \temp_1 < \temp_2 < ... <
\temp_M = 1$ such that
\begin{equation*}
  M \le 2n^{-1/2}| \E q_n(1) - \E q_n(0)|,
\end{equation*}
and
\begin{equation*}
 |\E q_n(\temp_{i+1}) - \E q_n(\temp_i)| \le n^{1/2} ,\quad  1 \le i \le M-1.
\end{equation*}
To achieve this, we can choose $\temp_i$ one by one, starting with $i=1,2$.  
Define the event $\Lambda(x) = \{ |q_n(\temp_i) - \E q_n(\temp_i)| \le x,  \  1 \le i
\le M \}$.  Then by~(\ref{eq:concentratin-for-q}),
\begin{equation}\label{eq:probability-of-Lambda-u}
\Pp \big(  \Lambda(x) \big) \ge 1 - M \cdot b_1  \exp \bigg\{  - b_2 \frac{x^2}{ n \ln^2n}
\bigg\}, \quad x \in (b_3, n \ln n].
\end{equation}

For $\omega \in \Lambda(x)$ and $\temp \in [\temp_i, \temp_{i+1}]$, since $q_n(\temp)$ and
$\E q_n(\temp)$ are both monotone in~$\temp$,
\begin{align*}
  |q_n(\temp) - \E q_n(\temp)| &= |\tilde{p}_n(\temp) - \E \tilde{p}_n(\temp)|\\
  & \le  |q_n(\temp_i) - \E q_n(\temp_{i+1})|  \vee  | q_n(\temp_{i+1}) - \E q_n(\temp_i)| \\
  &\le x + |\E q_n(\temp_i) - \E q_n(\temp_{i+1})| \\
  &\le x + n^{1/2}.
\end{align*}
Combined with~(\ref{eq:probability-of-Lambda-u}), this implies that
\begin{equation}
  \label{eq:uniform-concentration-for-p-tilde}
  \Pp \Big\{  |\tilde{p}_n (\temp) - \E \tilde{p}_{n}(\temp) | \le x + n^{1/2}, \temp \in [0,1] \Big\} \ge 1 - M \cdot b_1  \exp \bigg\{  - b_2 \frac{x^2}{ n \ln^2n}
\bigg\},
\end{equation}
for all~$x \in (b_3, n \ln n]$.

By Lemma~\ref{lem:truncated-log-partition-function-expectation}
and~(\ref{eq:error-estimate}), we have
\begin{equation}
  \label{eq:expectationp-tilde-from-an}
  | \E \tilde{p}_n(\temp) - \alpha_{0;\temp} n| \le D_1 + b_3 n^{1/2}\ln^2n, \quad \temp \in [0,1].
\end{equation}
This and Lemma~\ref{lem:exponential-bound-for-E} imply
\begin{align*}
  |\E q_n(1) - \E q_n(0) |
  &\le |\E \tilde{p}_n(1) - \E \tilde{p}_n(0)| + |E^{0,n}_{\le R_1}| \\
  &\le 2(D_1 + b_3 n^{1/2} \ln^2n) + n |\alpha_{0;1} - \alpha_{0;0}| + d_3 n \\
  &\le Kn.
\end{align*}
Hence $M \le2 Kn^{1/2}$.  Using this upper bound on~$M$
and~(\ref{eq:uniform-concentration-for-p-tilde}), (\ref{eq:expectationp-tilde-from-an}), we
complete the proof.
\end{proof}

Next we turn to the proof of Lemma~\ref{lem:continuity-of-alpha}.
\begin{lemma}
  \label{lem:doubling-argument-inequality}
  There is positive constant $b_5$ such that for all $\temp \in [0,1]$ and sufficiently large $n$, 
  \begin{equation}
    \label{eq:doubling-argument-inequality}
    |  \E p_{2n}(\temp) - 2 \E p_n(\temp)| \le b_5 n^{1/2} \ln^2n.
  \end{equation}
\end{lemma}

\begin{proof}
  Since $p_n(\cdot)$ is continuous, it suffices to show~(\ref{eq:doubling-argument-inequality})
  i.e., 
\begin{equation*}
    |\E \temp\ln Z_{0, 0; \temp}^{0,2n} - 2\E  \temp\ln Z_{0, 0; \temp}^{0,n} | \le b_5 n^{1/2}
    \ln^2 n, 
  \end{equation*}
   for~$\temp \in (0,1]$, and then use continuity of $\E p_n(\cdot)$.

  For $R_1$ introduced in Lemma~\ref{lem:estimates-by-comparing-action}, define
  \begin{align*}
  B &= \{ \gamma:   \max_{1 \le i \le 2n-1}|\gamma_i| \le 2R_1n \}, \\
  C &= \{ \gamma : |\gamma_n-\gamma_{n+1}| \le R_1\sqrt{2n},\
      |\gamma_n-\gamma_{n-1}| \le R_1\sqrt{2n}  \}.
\end{align*}
Since $E^{0,2n}_{ \le R_1} \subset B \cap C$,
Lemma~\ref{lem:truncated-log-partition-function-expectation} implies that 
\begin{equation}
\label{eq:truncation_does-not-change-much}
  |\E \temp \ln Z_{0, 0; \temp}^{0,2n} (B\cap C) - \E \temp \ln Z_{0,0; \temp}^{0, 2n}| \le D_1.
\end{equation}

To prove the lemma, we need to bound $\E \temp\ln Z_{0, 0; \temp}^{0,2n}(B\cap C)$ from
above and from below using $2\E \temp \ln Z^{0, n}_{0, 0; \temp}$ plus some error terms.
First, let us deal with the lower bound.
By the definition of the sets~$B$ and~$C$, we have
\begin{align*}
  Z_{0, 0; \temp}^{0,2n}(B\cap C)
\ge Z_{0, 0; \temp}^{0,2n}(B\cap C \cap \{\gamma_n \in [0,1) \}). 
\end{align*}
Let us now compare the action of every path $\gamma$ in $B\cap C \cap \{\gamma_n \in [0,1) \}$ to the action
of the modified path $\bar \gamma$ defined by $\bar{\gamma}_n = 0$ and 
$\bar{\gamma}_j = \gamma_j$ for $j \neq n$. We recall that the action of a path was defined in~(\ref{eq:def-of-action}).
Since $|\gamma_{n+1} - \gamma_{n}| \le R_1\sqrt{2n}$, $|\gamma_n-\gamma_{n-1}| \le R_1 \sqrt{2n}$, and $|\gamma_n| \le 1$,
we get
\begin{equation*}
\begin{split}
  |A^{0,2n}(\gamma) - A^{0,2n}(\bar \gamma)|
   &\le \frac{1}{2}\left| (\gamma_{n+1}-\gamma_n)^2-\gamma_{n+1}^2 +
    (\gamma_{n-1}-\gamma_n)^2 - \gamma_{n-1}^2
  \right|
  + 2F_{n,\omega}^{*}(0)
  \\ &\le 2R_1 \sqrt{2n} + 1 + 2F_{n,\omega}^{*}(0).
\end{split}
\end{equation*}
So, there is a constant $K_1>0$ such that 
\begin{equation}
\label{eq:Z-almost-a-product}
  Z_{0,0; \temp}^{0,2n}(B\cap C)\ge  Z_{0, 0; \temp}^{0,n}(D^-)Z_{0, 0; \temp}^{n,2n}(D^+) e^{-  \be \big(  K_1
    \sqrt{n} - 2F_{n,\omega}^{*}(0)  \big)},
\end{equation}
where
\begin{align*}
      D^- &= \{ \gamma: |\gamma_{n-1}| \le R_1 \sqrt{2n} + 1,\
        |\gamma_i| \le 2R_1n,\ 1 \le  i \le n-1\}, \\
  D^+ &= \{\gamma: |\gamma_{n+1}| \le R_1 \sqrt{2n} + 1,\
        |\gamma_i| \le 2R_1n,\ n+1 \le  i \le 2n-1 \}.
\end{align*}
Since~$E_{\le R_1}^{0,n} \subset D^-$ and~$E_{\le R_1}^{n,2n} \subset D^+$, Lemma~\ref{lem:truncated-log-partition-function-expectation} implies that
\begin{equation*}
  \temp  |\E \ln Z^{0, n}_{0, 0; \temp}(D^-) - \E \ln Z^{0, n}_{0, 0; \temp} |\le  D_1, \quad
  \temp |\E \ln Z_{0, 0; \temp}^{n,2n}(D^{+}) - \E \ln Z_{0, 0; \temp}^{n, 2n} | \le D_1.
\end{equation*}
Combining this with~\eqref{eq:Z-almost-a-product}, we obtain
\begin{align*}
\temp \E \ln  Z_{0, 0; \temp}^{0,2n}(B \cap C)
  &\ge \temp\Big(  \E \ln Z_{0, 0; \temp}^{0,n}(D^-) + \E \ln Z_{0, 0; \temp}^{n,2n}(D^+) \Big) -K_1 \sqrt{n} - 2\E F_{n,\omega}^{*}(0)
     \\
  &\ge  \temp\cdot 2\E \ln Z_{0, 0; \temp}^{0,n} - 2D_1-K_1 \sqrt{n} - 2\E F_{n,\omega}^{*}(0),
\end{align*}
where we used $\ln Z^{0, n}_{0,0; \temp} \disteq \ln Z_{0,0; \temp}^{n,2n}$ in the last inequality.

Next, let us turn to the upper bound.
 Similarly to~\eqref{eq:Z-almost-a-product}, we compare actions of generic paths in $B\cap C$  to the actions of the modified paths with integer
value at time $n$:
\begin{align*}
  Z_{0, 0; \temp}^{0,2n}(B\cap C)
  &= \sum_{k=-2R_1n}^{2R_1n-1}  Z_{0, 0; \temp}^{0,2n}(B\cap C\cap \{\gamma_n \in [k,k+1) \} ) \\
  &\le \sum_{k=-2R_1n}^{2R_1n-1} Z_{\temp}^{0,n}(0,k)Z_{\temp}^{n,2n}(k,0) e^{ \be \Big(  K_1 \sqrt{n} +
    2F_{n,\omega}^{*}(k)  \Big)} \\
  &\le 4R_1n \max_k [Z_{\temp}^{0,n}(0,k)Z_{\temp}^{n,2n}(k,0)]
    e^{\be \Big( K_1 \sqrt{n} + 2\max_k F_{n,\omega}^{*}(k) \Big)}, 
\end{align*}
where the maxima are taken over $-2R_1n \le k \le 2R_1n-1$.
Taking logarithm and then expectation of both sides, we obtain
\begin{align*}
  &\quad\temp \E \ln Z_{0, 0; \temp}^{0,2n}(B\cap C) \\
  &\le \temp \Big( \E \max_k \ln Z_{\temp}^{0,n}(0,k) + \E \max_k \ln Z_{\temp}^{n,2n}(k,0) \Big)
    + \temp \ln(4R_1n) + K_1 \sqrt{n} + 2 \E \max_k F_{n,\omega}^{*}(k) \\
  &\le   \max_k \E\temp \ln Z_{\temp}^{0,n}(0,k) + \E \max_k X_k
    + \max_k \E \temp\ln Z_{\temp}^{n,2n}(k,0) + \E \max_k Y_k    + K_2(\ln n + \sqrt{n} + 1) \\
  &\le 2\E \temp \ln Z_{0, 0; \temp}^{0,n} + \E \bigl[\max_k X_k+\max_k Y_k\bigr] \Big)+ K_2 (\ln n + \sqrt{n} + 1),
\end{align*}
for some constant $K_2>0$, where
\begin{equation*}
  X_k =\temp\Big( \ln Z_{\temp}^{0,n}(0,k) - \E\ln Z_{\temp}^{0,n}(0,k) \Big), \quad
  Y_k = \temp\Big( \ln Z_{\temp}^{n,2n}(k,0) - \E\ln Z_{\temp}^{n,2n}(k,0) \Big).
\end{equation*}
 In the second inequality, we used~\eqref{eq:item:4} to
 conclude 
\begin{equation*}
  \E \max_{-2R_1n\le k \le 2R_1n-1}F_{n,\omega}^{*}(k) \le b \ln (2n) + 4/\eta,
\end{equation*}
and in the third inequality, we used the fact that
\begin{equation*}
  \E \ln Z_{\temp}^{0,n}(0,k) \le \E \ln Z_{0, 0; \temp}^{0,n}, \quad
  \E \ln Z_{\temp}^{n,2n}(k,0) \le \E \ln Z_{0, 0; \temp}^{n,2n} = \E \ln Z_{0, 0; \temp}^{0,n}.
\end{equation*}
It remains to bound $\E \max_k X_k$ and $\E \max_k Y_k$.
By the shear invariance, all $X_k$ and~$Y_k$  have the same distribution, so
\begin{equation*}
  \E X_n^2 = \E Y_n^2 = \E \Big(\temp\ln Z_{\temp}^n\Big)^2 \le M(2)n^2
\end{equation*}
by Lemma~\ref{lem:second-moment-growth-of-partition-function}.
Let
\begin{equation*}
  \Lambda = \left\{ \max_k X_k \le rn^{1/2}\ln^{3/2} n, \quad \max_k Y_k \le rn^{1/2}\ln^{3/2}n \right\},
\end{equation*}
with $r$ to be determined. We have
\begin{align*}
  \E \left[\max_k X_k + \max_k Y_k\right]
  &\le \E \ONE_{\Lambda} (\max_kX_k + \max_kY_k) + \E \ONE_{\Lambda^{c}} (\max_kX_k+\max_kY_k) \\
  & \le 2rn^{1/2}\ln^{3/2}n + \sqrt{2\Pp(\Lambda^c) \E( \max_kX_k^2+\max_kY_k^2)} \\
  & \le 2rn^{1/2} \ln^{3/2}n + \sqrt{16\Pp(\Lambda^c) M(2)R_1n^3}.
\end{align*}
To bound the second term by a constant, we use Lemma~\ref{lem:free-energy-concentration-around-expectation}:
\begin{align*}
  \Pp (\Lambda^c) &\le \sum_{k=-2R_1n}^{2R_1n-1} \Biggl[\Pp \left\{ \temp|\ln Z_{\temp}^{0,n}(0,k) - \E\ln Z_{\temp}^{0,n}(0,k)| \ge rn^{1/2}\ln^{3/2}n \right\}  \\
  &\quad + \Pp \left\{ \temp|\ln Z_{\temp}^{n,2n}(k,0) - \E\ln Z_{\temp}^{n,2n}(k,0)| \ge rn^{1/2} \ln^{3/2}n \right\} \Biggr]\\
                  &\le 8R_1n \Pp \left\{\temp |\ln Z^n_{\temp} - \E\ln Z_{\temp}^{n}| \ge rn^{1/2}\ln^{3/2}n \right\} \\
  & \le 8R_1n b_1 \exp \{-b_2 r^2 \ln n \},
\end{align*}
and choose $r$ to ensure $b_2r^2>4$. This completes the proof.
\end{proof}

We can now use the following straightforward adaptation of~Lemma~4.2 of~\cite{HoNe}  from real argument functions to sequences:
\begin{lemma}
\label{lem:doubling-argument-lemma}
Suppose that number sequences $(a_n)$ and $(g_n)$ satisfy the following conditions: $a_n/n \to \nu$ as $n
\to \infty$, $|a_{2n} - 2a_n| \le g_n$ for  $n \ge n_0$ and $\lim_{n\to \infty} g_{2n}/g_n = \psi < 2$.
Then for any $c > 1/(2-\psi)$ and for~$n\ge n_1 = n_1\big( n_0,  (g_n), c \big)$,
\begin{equation*}
|a_n - \nu n | \le c g_n.
\end{equation*}
\end{lemma}

\begin{proof}
  Let $b_n = a_n/n$, $h_n = g_n/(2n)$.
  Then $|b_{2n} - b_n| \le h_n$ for $n>n_0$ and $\lim_{n\to \infty } h_{2n}/h_n = \psi/2$.

  Since $\psi/2 \le 1 - \frac{1}{2c} $, there is $N>n_0$ such that $h_{2m}/ h_m \le 1 -\frac{1}{2c}$ for all
  $m > N$.
  Let us now fix $n > N$.  Then for $k \ge 0$ we have $h_{2^kn} \le \bigl( 1-\frac{1}{2c} \bigr)^k h_{n}$.
  Therefore,
\begin{equation*}
  | b_n - b_{2^kn}| \le \sum_{i=0}^{k-1} |b_{2^{i+1}n} - b_{2^in} | \le \sum_{i=0}^{k-1} h_{2^in}
  \le 2c h_{n}.
\end{equation*}
We complete the proof by letting $k \to \infty$.
\end{proof}

\begin{proof}[Proof of Lemma~\ref{lem:continuity-of-alpha}]
Thanks to Lemma~\ref{lem:doubling-argument-inequality}, we can apply Lemma~\ref{lem:doubling-argument-lemma} to $a_n = \E p_n(\temp)$, $g_n = b_5 n^{1/2} \ln^2n$, $\nu = \alpha_{0; \temp}$, $\psi =\sqrt{2}$, and some fixed constant $c > 1/(2 - \psi)$ to obtain~(\ref{eq:error-estimate}).

The inequality~(\ref{eq:error-estimate}) implies that $\frac{1}{n}p_n(\temp)$ converge to
$\alpha_{0; \temp}$ uniformly for all $\temp \in [0,1]$.  Since for each $n\in \N$,
$\frac{1}{n}p_n(\cdot)$ is continuous and decreasing, the second part follows.
\end{proof}

\section{Straightness and tightness}
\label{sec:delta-straightness-and-tightness}

In this section, we modify our approach to straightness used in~\cite{2016arXiv160704864B}, obtaining
estimates that serve all $\temp\in(0,1]$ at the same time.  Also, we avoid using monotonicity, so the argument can be extended to higher dimensions.

\begin{theorem}
  \label{thm:all-space-time-point-compactness}
 There is a full measure set $\Omega'$ such that for every $\omega \in \Omega'$ the
 following holds:  if $(m,x) \in \Z \times \R$,  $v'\in \R$, and $0\le u_0 <u_1$,
  then there is a random constant
\begin{equation*}
n_0= n_0\big(\omega,m,[x], [|v'|+u_1], [(u_1-u_0)^{-1}] \big)
\end{equation*}
(where $[\cdot]$ denotes the integer part) such that 
\begin{equation}
  \label{eq:remote-control-all-space-time-all-space-time} 
    \mu_{x,\nu; \temp}^{m,N} \big\{ \gamma: |\gamma_{m+n}  - v'n| \ge u_1n \big\}
  \le  \nu \big( [ (v'-u_0)N, (v'+u_0)N]^c \big) + e^{-\be n^{1/2}}
\end{equation}
and 
\begin{multline}
  \label{eq:compactness-control-by-terminal-measure-all-space-time}
  \mu_{x,\nu; \temp}^{m,N} \big\{ \gamma: \max_{1 \le i \le n}|\gamma_{m+i} - v'i| \ge (u_1+R_1+1)n \big\} \\
  \le  \nu \big( [ (v'-u_0)N, (v'+u_0)N]^c \big) + 2e^{-\be n^{1/2}}
\end{multline}
hold true for any terminal measure $\nu$, $(N-m)/2 \ge n \ge n_0$, and all $\temp \in (0,1]$.
Here, we use~$R_1$ that has been introduced in Lemma~\ref{lem:estimates-by-comparing-action}.
\end{theorem}

Let us begin with a corollary of Theorem~\ref{thm:concentration-of-free-energy}. 
\begin{lemma}
\label{lem:substution-lemma}
Let $m, p, q \in \Z$ and~$n \in \N$.
If~$n$ is sufficiently large,  then on an event with probability at least $1 - e^{-n^{1/3}}$,
it holds that for all $x \in [p,p+1]$, $y \in [q,q+1]$, and~$\temp \in (0,1]$, 
\begin{equation*}
  \big|  \temp \ln Z^{m,m+n}_{x,y; \temp} - \alpha_{\temp}(n, x-y)  \big) 
  \big| \le n^{3/4}, 
\end{equation*}
where $\alpha_{\temp}(k,z) = \alpha_{\temp }(z/k)  \cdot k= \alpha_{0; \temp} k - \frac{z^2}{2k}$.
\end{lemma}
\begin{proof}
  Without loss of generality we can assume $m=0$ and $p = q
= 0$.  Taking $u = n^{3/4}/2$, by Theorem~\ref{thm:concentration-of-free-energy} we have that on an event $\Lambda_1$ with probability at
least~$1  - c_1e^{-c_2 \frac{n^{1/2}}{4\ln^2 n}}$, 
\begin{equation}\label{eq:at-0-0}
  \big|  \temp \ln Z^{0,n}_{0, 0; \temp} -  \alpha_{0; \temp}n
  \big| \le n^{3/4}/2,\quad \temp \in (0,1].
\end{equation}

We recall the constant~$R_1$ in Lemma~\ref{lem:estimates-by-comparing-action} and define the
following modification of~$Z_{x,y; \temp}^{0,n}$: 
\begin{align*}
  \bar{Z}_{x,y; \temp}^{0, n}
&  = \int_{|x_1|, |x_{n-1}| \le R_1 \sqrt{n} + 1 }  Z^{1,n-1}_{x_1,
  x_{n-1}; \temp} \, dx_1 dx_{n-1} \\
&  \qquad \cdot  \frac{1}{2\pi \cdot \temp} \exp \Big(  - \be \cdot \big[  \frac{(x_1 -
      x)^2}{2 } + \frac{(x_{n-1} - y)^2}{2} +  F_1(x_1) + F_n(y)\big]  \Big).
\end{align*}
For all $x,y \in [0,1]$, we have
\begin{align}
  \label{eq:cts-for-Z-bar}
&\quad  \temp | \ln \bar{Z}_{x,y; \temp}^{0,n} - \ln \bar{Z}_{0,0; \temp}^{0,n}| \\  \nonumber
  &\le \max_{y \in [0,1]} \big(  |F_n (0)| + |F_n(y)| \big)  + \max_{\substack{ x,y\in [0,1]  \\ |z|, |w|
  \le R_1 \sqrt{n} + 1 }} 
  \frac{1}{2} \big|  (z-x)^2 + (w - y)^2 - z^2 - w^2 \big| \\ \nonumber
& \le \max_{y \in [0, 1]}  \big(  |F_n (0)| + |F_n(y)| \big)+ 2R_1 \sqrt{n} + 3.
\end{align}
Using~(\ref{eq:polymer-measure-with-more-than-linear-action}) in
Lemma~\ref{lem:estimates-by-comparing-action} and the fact that 
\begin{equation*}
  \mu_{x,y; \temp}^{0,n} \big(  \big\{  \gamma: |\gamma_1| \vee |\gamma_{n-1}| > R_1 \sqrt{n} +1  \big\} \big)
  \le \mu^{0,n}_{x,y; \temp} ( \cup_{s \ge R_1}E_s^{0,n}),  \quad x, y \in [0,1], 
\end{equation*}
we obtain that on an event~$\Lambda_2$ with probability
at least~$1- 3e^{-d_1R_1 n}$, 
 \begin{equation}\label{eq:diff-between-Z-and-Z-bar}
\temp|\ln  \bar{Z}_{x,y; \temp}^{0,n} - \ln Z_{x,y; \temp}^{0,n}| \le \temp |\ln (1-2^{-\be \cdot R_1
  n})| \le  |\ln (1 - 2^{-R_1})|, \quad x,y \in [0,1],\ \temp \in (0,1].
\end{equation}
Due to assumption~\ref{item:exponential-moment-for-maximum} and Markov inequality, there is an
event~$\Lambda_3$ with probability at least~$1- e^{\varphi-\eta n^{3/4}/8}$ such that 
\begin{equation}\label{eq:F-is-small}
 \max_{x \in [0, 1]} |F_0(x)| \le n^{3/4}/8.
\end{equation}
Also, for all $x,y  \in [0, 1]$,  we have
\begin{equation}\label{eq:continuity-of-alpha}
  |\alpha_{0; \temp} n - \alpha_{\temp}(n, x-y)|
  = \frac{1}{2n} (x-y)^2 \le 1.
\end{equation}

Now consider the event $\Lambda = \Lambda_1 \cap \Lambda_2 \cap \Lambda_3$ and
combine~(\ref{eq:at-0-0}), (\ref{eq:cts-for-Z-bar}), (\ref{eq:diff-between-Z-and-Z-bar}),
(\ref{eq:F-is-small}), and~(\ref{eq:continuity-of-alpha}) together.  Then~$\Pp (\Lambda) \ge
1 - e^{-n^{1/3}}$ and if $\omega \in \Lambda$, then
\begin{equation*}
  | \temp \ln Z^{0,n}_{x,y; \temp} - \alpha_{\temp}(n,x-y) |
  \le \frac{n^{3/4}}{2} + 2 \cdot \frac{n^{3/4}}{8} + 2R_1\sqrt{n} + 4  + |\ln (1 - 2^{-R_1})|  \le n^{3/4}.
\end{equation*}
This concludes the proof.
\end{proof}

For $(m,x), (n,y) \in \Z \times \R$ with $m<n$, we define $[(m,x),(n,y)]$ to be the constant velocity path 
connecting $(m,x)$ and $(n,y)$, i.e.,  $[(m,x),(n,y)]_k=x+\frac{k-m}{n-m}(y-x)$ for~$k\in
[m,n]_{\Z} $.
For $(m,p), (n,q) \in \Z \times \Z$, we define the events
\begin{multline}
  \label{eq:def-of-event-A}
  A_{p,q}^{m,n} = \Big\{   \mu_{x,y; \temp}^{m,n} \big\{  \max_{k \in I(m,n)} | \gamma_k - [(m,p),
  (n,q)]_k \ge (n-m)^{8/9}
  \big\} \le e^{- \temp^{-1}(n-m)^{1/2}}, \\
  \ x \in [p,p+1], y \in [q,q+1],\ \temp\in (0,1]
  \Big\},
\end{multline}
where $I(m,n) = [\frac{3m+n}{4}, \frac{m+3n}{4}]_{\Z}$, and the events
\begin{multline}
  \label{eq:def-of-event-B}
  B_{p,q}^{m,n} = \Big\{  \mu_{x,y; \temp}^{m,n}  \big\{  \max_{k \in [m,n]_{\Z}} |\gamma_k - [(m,p),
  (n,q)]_k| \ge R_1n \big\} \le 2^{- \be R_1n},\\
  \ x \in [p,p+1], y \in [q,q+1],\ \temp \in [0,1]  \Big\},
\end{multline}
where $R_1$ is introduced in Lemma~\ref{lem:estimates-by-comparing-action}.
Such events $A_{p,q}^{m,n}$ and~$B_{p,q}^{m,n}$ are measurable since for a fixed Borel set $D \in
S^{-\infty,+\infty}_{*,*}$, $\mu_{x,y; \temp}^{m,n}(D)$ is continuous in~$x$, $y$ and~$\temp$.
Moreover, by translation and
shear invariance, the probability of  $A_{p,q}^{m,n}$ and $B_{p,q}^{m,n}$  depends only on $n-m$.

The probability of~$B_{p,q}^{m,n}$ can be estimated using~(\ref{eq:large-deviation-for-Z-bar}) in
Lemma~\ref{lem:estimates-by-comparing-action}.
The following lemma gives estimation on the probability of $A_{p,q}^{m,n}$.
\begin{lemma}\label{lem:concentration-of-polymer-measure-in-delta-rectangle}
For some constant $k_1$, if $N$ is large enough, then
  \begin{equation*}
    \Pp \big(  A_{0,0}^{0,N} \big) \ge 1- k_1 N^2e^{-N^{1/3}}.
  \end{equation*}
\end{lemma}

\begin{proof}
By~(\ref{eq:large-deviation-for-Z-bar}) in Lemma~\ref{lem:estimates-by-comparing-action}, there is an event $\Lambda_1$ with $\Pp(\Lambda_1)\ge{1-3e^{-d_1 R_1 N}}$ on which the following holds: 
\begin{equation}\label{eq:not-likely-outside-linear-box}
  \mu_{x,y; \temp}^{0,N} \big(  \{ \gamma: \max_{1 \le k \le N-1} |\gamma_k| \le R_1 N \} \big) \le
  2^{-\be R_1 n}, \quad
  x,y\in [0,1],\ \temp \in (0,1].
\end{equation}
Applying Lemma~\ref{lem:substution-lemma} with~$(m,n,p,q)$ running over the set
\begin{equation*}
  \{ (0,k,0,l): k \in [  \textstyle{\frac{N}{4},\frac{3N}{4}}], |l| \le R_1 N \} \cup
  \{ (k,N-k, l, 0): k \in [ \textstyle{\frac{N}{4}, \frac{3N}{4} }], |l| \le R_1N\},
\end{equation*}
we can obtain an event~$\Lambda_2$ with probability at least $1 - C_1 N^2 e^{- N^{1/3}}$
on which the following holds for all $x, y \in  [0,1]$: 
\begin{align}\label{eq:substitution}
  | \temp \ln Z^{0,k}_{x,z; \temp} - \alpha_{\temp}(k, z-x) | \le k^{3/4} \le  N^{3/4} ,
  & \quad k \in [ \textstyle{\frac{N}{4},\frac{3N}{4}}],\
    |z| \le R_1N,   \\ \nonumber
  | \temp \ln Z^{k,N}_{z,y; \temp} - \alpha_{\temp}(N-k,y-z) | \le (N-k)^{3/4}\le N^{3/4},
  &\quad k \in [\textstyle{\frac{N}{4},\frac{3N}{4}}],\ 
    |z| \le R_1N ,\\ \nonumber
  |\temp \ln Z^{0,N}_{x,y; \temp} - \alpha_{\temp}(N, x-y) | \le  N^{3/4}.
\end{align}
Using~(\ref{eq:substitution}), for~$\omega\in \Lambda_2$, all $k \in \big[ \frac{N}{4}, \frac{3N}{4} \big]$ and all $x,y \in
[0,1]$, we have
\begin{align*}
& \qquad  \mu_{x, y; \temp}^{0,N} \big(  \{ \gamma: |\gamma_k| \in [N^{8/9}, R_1N] \} \big) \\
  &= \big(  Z_{x,y; \temp}^{0,N} \big)^{-1} \int_{|z| \in [N^{8/9}, R_1N]} Z^{0, k}_{x,z; \temp}
    Z^{k,N}_{z,y; \temp} \, dz \\
  &\le \exp \big({ \textstyle \be \big[ 3N^{3/4} +\frac{(x-y)^2}{2N} } \big] \big) \int_{|z| \in [N^{8/9} , R_1 N]}
    \exp \big(- { \textstyle  \be \big[  \frac{(x-z)^2}{2k} + \frac{(y-z)^2}{ 2(N-k)} \big]}   \big) \, dz\\
  &\le \exp \big( {\textstyle \be  \big[ 3N^{3/4} + 1 \big]} \big)  \int_{|z| \ge N^{8/9}/2}
    \exp\big(  - { \textstyle \be \frac{2z^2}{N}} \big)\, dz\\
  &\le N^{1/9} \exp \big( - {\textstyle \be} \big[ N^{7/9}/2 - 1 - 3N^{3/4} ]\big),
\end{align*}
where in the last inequality we use the following bound on the tail of Gaussian integral: for $a, b
> 0$,
\begin{equation*}
  \int_{|x| \ge b } e^{-\frac{x^2}{a}} dx
  \le \frac{a}{b} e^{-\frac{b^2}{a}}.
\end{equation*}
Combining this with~(\ref{eq:not-likely-outside-linear-box}), we can conclude that
$A^{0,n}_{0,0} $ is included in $\Lambda_1 \cup \Lambda_2$, which has probability at least
$1-C_2N^2e^{-N^{1/3}}$.  Here, the constants $C_1$ and~$C_2$ are independent of $N$.
This completes the proof.
\end{proof}

\begin{lemma}
\label{lem:high-probability-straightness}
Let $c > 0$, $0 < v_0 < v_1$, $v' \in \R $ and~$m,p \in \Z$.  Suppose $|v'| + v_1 < c$.
There are constants $n_1 = n_1 ([|v_1 - v_0|^{-1}])$ and $k_2$ such that when $n
> n_1$, there is an event~$\Omega^{(1)}_{c,n}(m,p)$
with probability at least $1 - k_2cn^3e^{-n^{1/3}}$ on which the following holds: for all $N > 2n$,
$\temp \in (0,1]$, $x \in [p,p+1]$ and for any terminal
measure $\nu$,
\begin{multline}
    \label{eq:remote-control} 
  \mu_{x,\nu; \temp}^{m,m+ N} \pi^{-1}_{m+n}   \big(  [p+  (v'-v_1)n, p+ (v'+v_1)n]^c \big)\\
  \le \nu \big(  [p+ (v'-v_0)N, p + (v'+v_0)N]^c \big) + e^{-\be n^{1/2}}, 
\end{multline}
and
\begin{multline}
  \label{eq:compactness-control-by-terminal-measure}
  \mu_{x,\nu; \temp}^{m,m+N} \big\{ \gamma: \max_{1 \le i \le n}|\gamma_{m+i} - p - v'i| \ge (v_1+R_1+1)n \big\} \\
  \le \nu \big( [ p+(v'-v_0)N, p+(v'+v_0)N]^c \big) + 2e^{-\be n^{1/2}}.
\end{multline}
\end{lemma}

\begin{proof}
We will choose $\Omega^{(1)}_{c,n}(m,p) = \theta^{m,p} \Omega^{(1)}_{c,n}$ ($\theta$ is the
space-time shift), where
\begin{equation}\label{eq:definition-of-Omega-1}
  \Omega^{(1)}_{c,n}= \Bigg( \bigcap_{ \substack{j \ge 2n\\ |q| \le (c+1)j}} A_{0,q}^{0,j}   \Bigg)
  \cap \Big(  \bigcap_{|q| \le (c+1)n} B_{0,q}^{0,n} \Big).
\end{equation}
Due to~(\ref{eq:large-deviation-for-Z-bar}) in Lemma~\ref{lem:estimates-by-comparing-action}, $\Pp
(B^{0,n}_{0,q}) \ge 1 - 3e^{-d_1R_1n}$.  This and
Lemma~\ref{lem:concentration-of-polymer-measure-in-delta-rectangle} imply that~${\Pp (\Omega^{(1)}_{c,n})
  \ge 1 - k_2 cn^3 e^{-n^{1/3}}}$ for some constant $k_2$.

Without loss of generality, we will assume $(m,p) = (0,0)$.  In showing~(\ref{eq:remote-control})
and~(\ref{eq:compactness-control-by-terminal-measure}), we will also assume $v'=0$ for simplicity.
The extension to other values of $v'$ is straightforward.
Let us fix a terminal measure $\nu$ and $\temp \in (0,1]$, $x \in [0,1]$, $N \ge 2n$, and assume $\omega
\in \Omega_{c,n}^{(1)}$.

For~(\ref{eq:remote-control}), it suffices to show that if~$n$ is large, then
\begin{equation*}
\mu_{x,\nu; \temp}^{0,N} \big(  \{ \gamma: |\gamma_N | <N v_0, |\gamma_n | \ge nv_1  \} \big) <
e^{- \be n^{1/2}}. 
\end{equation*}
Let $k$ be the unique integer such that~${2^k n \le N  < 2^{k+1}n}$.  For~$l \in [0,k]_{\Z}$, define
\begin{equation*}
i_l  =
\begin{cases}
  n\cdot 2^l, & 0 \le l \le k-1,      \\
  N, & l = k.
\end{cases}
\end{equation*}
Let us consider the following inequality that appears in the definition of $A_{0,[\gamma_{i_l}]}^{0,i_l}$: 
\begin{equation}
\label{eq:1}
\bigg|   \big[ (0,0), (i_l, [\gamma_{i_l}])  \big]_{i_{l-1}} - \gamma_{i_{l-1}} \bigg|
 = \bigg|  [\gamma_{i_l}] \cdot \frac{i_{l-1}}{i_l} -  \gamma_{i_{l-1}}\bigg|
\le (i_l)^{8/9}.
\end{equation}
If a path $\gamma$ satisfies~(\ref{eq:1}) for all~$l \in [l'+1, k]_{\Z}$, then
\begin{equation}
  \label{eq:slope-difference-estimate}
  \begin{split}
&\quad  \bigg|  \frac{\gamma_{i_l'}}{i_l'} - \frac{\gamma_N}{N}
  \bigg|   \le \sum_{l = l'+1}^{k} \frac{ (i_l)^{8/9} + 1}{i_{l-1}} \\
  &\le n^{-1/9} \Big[  \sum_{l=l'+1}^{k-1} \big(  2^{8/9} \cdot 2^{-\frac{1}{9}(l-1)} + 2^{-(l-1)} \big)
    +  \big(2^{16/9}\cdot  2^{-\frac{1}{9}(k-1)} + 2^{-(k-1)}  \big)
    \Big]\\
  &\le K_1 n^{-1/9}
\end{split}
\end{equation}
for some absolute constant $K_1$.

For $l' \in [0, k-1]_{\Z}$, let us define the set of paths
\begin{equation*}
\Lambda_{l'} = \{ \gamma: \text{ (\ref{eq:1}) holds for all $l \in[l'+1, k]_{\Z}$ and  $|\gamma_N| <
  Nv_0$. } \}.
\end{equation*}
We also define $\Lambda_k = \{ \gamma: |\gamma_N| < Nv_0 \}$.
Suppose $n \ge \Big(  \frac{K_1}{|v_1-v_0|  \wedge (1/2)} \Big)^9$.
If a path~${\gamma \in \Lambda_{l'} \setminus \Lambda_{l'-1}}$ ($l \in [1,k]_{\Z}$),
then~(\ref{eq:slope-difference-estimate}) implies~${|\gamma_{i_{l'}}| < (c+1/2)i_{l'}}$.
Therefore,
\begin{align*}
  \mu_{x,\nu; \temp}^{0,N} \Big( \Lambda_{l'} \setminus \Lambda_{l'-1} \Big)
  &= \int \nu(dz)  \big(  Z^{0,N}_{x,z; \temp} \big)^{-1} \int_{-(c+1/2)i_{l'}}^{(c+1/2)i_{l'}} dw \, 
    Z^{0,i_{l'}}_{x,w; \temp}(\Lambda_{l'} \setminus \Lambda_{l'-1})  Z^{i_{l'}, N}_{x,w; \temp}(\Lambda_{l'} \setminus \Lambda_{l'-1}) \\
  &\le \int \nu(dz)  \big(  Z^{0,N}_{x,z; \temp} \big)^{-1}  \int_{-
    (c+1/2)i_{l'}}^{(c+1/2)i_{l'}} dw \
   e^{- \be (i_{l'})^{1/2}} Z^{0,i_{l'}}_{x,w; \temp} Z^{i_{l'}, N}_{x,w; \temp} (\Lambda_{l'} \setminus
    \Lambda_{l'-1})\\
  &\le e^{-\be (i_{l'})^{1/2}}.
\end{align*}
Here, in the second inequality we used that $\omega \in \Omega^{(1)}_{c,n} \subset
  A_{0,[w]}^{0,i_{l'}}$ for $|w| \le (c+1/2)i_{l'}$, and hence 
\begin{equation*}
  \mu_{x,w; \temp}^{0,i_{l'}} (\Lambda_{l'} \setminus \Lambda_{l'-1})
    \le e^{-\be (i_{l'})^{1/2}}.
\end{equation*}
Also, $|v_0-v_1| > K_1n^{-1/9}$ (which holds for large $n$) and~(\ref{eq:slope-difference-estimate}) imply that 
\begin{equation*}
\Lambda_0 \cap \big\{ \gamma: |\gamma_n| > n v_1\big\} = \varnothing.
\end{equation*}
Combining all these estimates, we have 
\begin{align*}
  \mu_{x,\nu; \temp}^{0,N} \big(  \{ \gamma: |\gamma_N | <N v_0, |\gamma_n | \ge nv_1  \} \big)
 &\le  \sum_{l'=1}^k \mu_{x,\nu; \temp}^{0,N} \big(  \Lambda_{l'}\setminus \Lambda_{i'-1} \big)
  \le \sum_{l'=1}^k e^{-\be (i_{l'})^{1/2}}\\
  & \le \sum_{m=2n}^{\infty} e^{ -\be m^{1/2}} \le 
  e^{- \be n^{1/2}},
\end{align*}
which completes the proof of~(\ref{eq:remote-control}).

Now we turn to~(\ref{eq:compactness-control-by-terminal-measure}).
Let 
\begin{equation*}
D = \big\{ \gamma : \ \max_{1 \le i \le n} |\gamma_i| \ge (v_1+R_1+1)n \big\}.
\end{equation*}
If $|z| \le v_1n$,  then 
\begin{equation*}
\mu_{x,z; \temp}^{0,n} (D )
\le
\mu_{x,z; \temp}^{0,s}  \{ \gamma: \max_{1 \le i \le n} | \gamma_i - [(0,0), (n,[z])]_i | \ge R_1n\}.
\end{equation*}
For all $|z| \le v_1s $, since $\omega \in B_{0,[z]}^{0,n}$, we have  $\mu_{x,z; \temp}^{0,n} (D) \le 2^{- \be R_1n }$.
 Therefore,
\begin{equation*}
\mu_{x,\nu; \temp}^{0,N} \bigl(  D \cap \{ |\gamma_s| \le v_1s  \} \bigr) \le 2^{ - \be R_1 n }.
\end{equation*}
Then~(\ref{eq:compactness-control-by-terminal-measure}) follows from this and~(\ref{eq:remote-control}).
\end{proof}

\begin{proof}[Proof of Theorem~\ref{thm:all-space-time-point-compactness}]
The Theorem directly follows from Lemma~\ref{lem:high-probability-straightness} 
and the Borel--Cantelli Lemma.
\end{proof}

\section{Infinite volume polymer measures and their zero-temperature limit}
\label{sec:zero-temperature-limit}

In this section, we will prove our main results, Theorems~\ref{thm:invisid-limit-for-polymer-measures} and~\ref{thm:convergence-of-busemann-function}.
We will show that $\Omega'$ introduced in
Section~\ref{sec:delta-straightness-and-tightness} can be chosen as the full measure set the existence of which is claimed in Theorem~\ref{thm:invisid-limit-for-polymer-measures}, 
and that 
we can take $\hat{\Omega}_v = \Omega' \cap
\Omega_{v;0} \cap \bigcap_{\temp \in \mathcal{D}} \Omega_{v;\temp}$ in
Theorem~\ref{thm:convergence-of-busemann-function}.

Let us first prove part~\ref{item:existence-of-inifinte-polymer-measures} of
Theorem~\ref{thm:invisid-limit-for-polymer-measures}, since it only uses the
properties of~$\Omega'$ established in Section~\ref{sec:delta-straightness-and-tightness}.
We recall the following notion of tightness. For fixed~$(m,x) \in \ZR$, suppose $(\mu_k)$ is a family of 
probability measures such that for each $k$, $\mu_k$ is defined on $S^{m,N_k}_{x,*}$, for some $N_k \to
\infty$, as $k\to\infty$.  We say that $(\mu_k)$ is tight if for each $\eps > 0$, there is a compact set $K \subset
\R^n$ such that 
\begin{equation}
\label{eq:definition-of-tightness}
\mu_k \pi_{m,m+n}^{-1} (K^c) < \eps, \quad N_k > m+n.
\end{equation}

\begin{proof}[Proof of part~\ref{item:existence-of-inifinte-polymer-measures} in  Theorem~\ref{thm:invisid-limit-for-polymer-measures}]
Let $\mu_N = \mu^{m, N}_{x, Nv; \temp}$.  We will show that the family of measures $\big( \mu_N
\big)_{N > m}$ is tight and any limit point in the weak topology belongs to $\Pc^{m,
  +\infty}_x(v)$.

Given $\eps>0$, choosing $v' = v$, $u_0 = 0$, $u_1 = 1$ in
Theorem~\ref{thm:all-space-time-point-compactness}, we see that if 
\begin{equation*}
(N-m)/2 \ge n \ge n_0(\omega, m, [x] , [|v| + 1], 1) \vee \ln^2 \eps,
\end{equation*}
then, due to (\ref{eq:compactness-control-by-terminal-measure-all-space-time}), 
\begin{equation*}
  \mu_N\Bigl\{ \gamma: \max_{m \le i \le m+n} |\gamma_i -vi| \ge (R_1+2)n \Bigr\}
  \le \delta_{Nv} \big( [Nv, Nv]^c \big) + 2e^{- \be n^{1/2}}
  \le 2 \eps.
\end{equation*}
Therefore, $\big( \mu_N \big)_{N > m}$ is tight.

Suppose $\mu_N$ converge to $\mu$ weakly.  To show that $\mu \in \Pc^{m,+\infty}_x(v)$, it suffices to
show that for all $\eps> 0$, 
\begin{equation}
\label{eq:6}
\sum_{n = 1}^{\infty}
\mu \pi^{-1}_{m+n} \big(  [(m+n)(v-\eps), (m+n)(v+\eps)]^c \big) < \infty.
\end{equation}
Let us choose~$v'=u$, $u_0 = \eps/2$ and $u_1
= \eps$.  Then (\ref{eq:remote-control-all-space-time-all-space-time}) in
Theorem~\ref{thm:all-space-time-point-compactness} implies that if~$(N-m)/2 \ge n \ge n_0 (\omega, m, [x],
[|v| + \eps], [2\eps^{-1}])$, then
\begin{multline*}
  \mu_N \pi^{-1}_{m+n} \big(  [(m+n)(v-\eps), (m+n)(v+\eps)]^c \big)
  \\ \le \delta_{Nv} \big(  [N(v-\frac{\eps}{2}), N(v+\frac{\eps}{2})]^c \big) + e^{-\be n^{1/2}}
   e^{-\be n^{1/2}}.
\end{multline*}
Taking $N\to\infty$, by weak convergence we have 
\begin{equation*}
\mu \pi^{-1}_{m+n}\big( [(m+n)(v-\eps),(m+n)(v+\eps)]^c \big) \le e^{-\be n^{1/2}}.
\end{equation*}
This implies (\ref{eq:6}) and completes the proof.
\end{proof}

Now we proceed to prove the rest of Theorem~\ref{thm:invisid-limit-for-polymer-measures} and
Theorem~\ref{thm:convergence-of-busemann-function}.

Let us fix $\omega \in \Omega'$ and $(m,x) \in \ZR$.
Let $\mu_{\temp} \in \Pc^{m,+\infty}_{x; \temp}(v)$, $\temp \in (0,1]$.  We first derive some properties for such a family $\big( \mu_{\temp} \big)$.

\begin{lemma}
\label{lem:tightness-IVPM-with-different-temp}
If $n > n_0(\omega, m, [x], [|v| + 1], 2)$, then for all $\temp \in (0,1]$,
\begin{equation}
\label{eq:tightness-IVPM-at-different-temp}
\mu_{\temp} \{ \gamma: \max_{m \le i \le m+n} |\gamma_i - vi| \ge (R_1 + 2)n \}
\le 2e^{- \be n^{1/2}}.
\end{equation}
\end{lemma}

\begin{proof}
  Applying Theorem~\ref{thm:all-space-time-point-compactness}
with $(v',u_0, u_1) = (v, 1/2, 1)$, when $(N-m)/2 > n$ we have
\begin{multline*}
  \mu_{\temp} \{ \gamma: \max_{m \le i \le m+n} |\gamma_i - vi| \ge (R_1 + 2)n \}
\\  = \mu_{x, \nu_N; \temp}^{m, N}\{ \gamma: \max_{m \le i \le m+n} |\gamma_i - vi| \ge (R_1 + 2)n
\} \\
\le \mu_{\temp} \pi_N^{-1} \big(  [N(v-1/2),N(v+1/2)]^c \big) + 2e^{-\be n^{1/2}}.
\end{multline*}
Since $\lim\limits_{N \to \infty}  \mu_{\temp}  \pi_N^{-1} \big(  [N(v-1/2), N(v+1/2)]^c \big) = 0$, (\ref{eq:tightness-IVPM-at-different-temp}) follows.
\end{proof}

\begin{lemma}
\label{lem:uniform-decay-for-marginal-measure}
For any~$\eps > 0$ and $\temp \in (0,1]$, if $n > n_0(\omega, m, [x], [|v| + \eps], [2\eps^{-1}])$, then 
\begin{equation}
  \label{eq:uniform-decay-for-marginal-measure}
  \mu_{\temp} \big(  [(m+n)(v-\eps), (m+n)(v+\eps)]^c \big)
  \le e^{-\be n^{1/2}}.
\end{equation}
\end{lemma}
\begin{proof}
  The proof is similar to that of Lemma~\ref{lem:tightness-IVPM-with-different-temp}.
\end{proof}

\begin{lemma}
\label{lem:approximate-by-finite-support-terminal-measures}
There are a constant $c>0$ and terminal measures~$\big( \nu^N_{ \temp} \big)_{N > m,\ \temp \in (0,1]}$
satisfying 
  \begin{equation}
\label{eq:finite-support-assumption}
\nu^N_{\temp} \big( [-cN, cN]^c \big)=0, \quad N > m \vee 0,\ \temp \in (0,1], 
\end{equation}
such that for each $\temp$, $\mu_{\temp}$ is the weak limit of $\mu_{x, \nu^N_{\temp}; \temp}^{m,
  N}$ as $N \to \infty$.
\end{lemma}

\begin{proof}
Let us define $\nu_{\temp}^N$ as follows:
\begin{equation*}
  \nu_{\temp}^N(A)  = \big( D_{\temp}^N \big)^{-1}
\mu_{\temp} \pi_N^{-1} \big( A \cap B_N \big),  \quad A \subset \mathcal{B}(\R),
\end{equation*} 
where $B_N = [ N(v-1), N(v+1)]$ and~$D_{\temp}^N =  \mu_{\temp} \pi_N^{-1} \big( B_N ) $.
For any $n > m$ and any Borel set $\Lambda \subset \mathcal{B}(\R^{n-m})$, we have
\begin{align*}
& \quad|  \mu_{\temp}\pi_{m,n}^{-1} (\Lambda) - \mu_{x, \nu^N_{\temp}; \temp}^{m,  N}
  \pi_{m,n}^{-1}(\Lambda)| \\
  &\le     |  \mu_{\temp}\pi_{m,n}^{-1} (\Lambda) - D_{\temp}^N\mu_{x, \nu^N_{\temp}; \temp}^{m,  N}
    \pi_{m,n}^{-1}(\Lambda)|
    + (1 - D_{\temp}^N ) \mu_{x, \nu^N_{\temp}; \temp}^{m,  N}    \pi_{m,n}^{-1}(\Lambda) \\
  &\le \nu_{\temp}^N(B^c_N)  + (  1 - D_{\temp}^N )  \mu_{x, \nu^N_{\temp}; \temp}^{m,  N}
    \pi_{m,n}^{-1}(\Lambda). 
\end{align*}
The right hand side goes to zero, since $\mu_{\temp} \in \Pc^{m,+\infty}_x(v)$ implies that 
\begin{equation*}
1 - D_{\temp}^N =  \nu_{\temp}^N(B^c_N) = \mu_{\temp}\pi^{-1}_N \big(  [N(v-1), N(v+1)]^c \big) \to
0, \quad N \to \infty.
\end{equation*}
This shows that $\mu_{\temp}$ is the weak limit of $\mu_{x, \nu^N_{\temp}; \temp}^{m,N}$ and
completes the proof.
\end{proof}

Let us recall that the {\it locally uniform} (LU) topology on $C(\R^{d})$ is defined by the metric
\[
 d(f,g)=\sum_{k=1}^{\infty} 2^{-k} \left(1\wedge \sup_{|x|\le k} |f(x)-g(x)|\right),\quad f,g\in C(\R^{d}).
\]
Convergence in this metric (also called {\it LU-convergence}) is equivalent to uniform convergence on
every compact subset of $\R^{d}$. LU-precompactness of a family $(f_n)$ is equivalent to equicontinuity
and uniform boundedness of $(f_n)$ on every compact set.

Before we continue on properties of $\big( \mu_{\temp} \big)$, let us state the following lemma, whose proof will be given at the end of this section.
\begin{lemma}
  \label{lem:tightness-of-log-density}
  Let $\omega \in \Omega'$ and $m, n \in \Z$ with $m < n$.
  Suppose a family of probability measures $\bigl( \nu_{\temp}^N \bigr)_{N > n, \temp \in (0,1]}$
satisfies~(\ref{eq:finite-support-assumption}) for some constant $c$. 
For $n < N$, let~$f^N_{m,n; \temp}(x,\cdot)$ be the density of $\mu_{x, \nu_{\temp}^N; \temp}^{m,N} \pi_n^{-1}$, namely,
\begin{equation*}
f^N_{m,n; \temp}(x,y) = \int_{-cN}^{cN} \frac{Z^{m,n}_{x,y; \temp} Z^{n, N}_{y, z;
    \temp}}{Z^{m,N}_{x,z; \temp}}\,  \nu^N_{\temp} (dz).
\end{equation*}
Then,  $\Big(  \temp \ln f^N_{m,n; \temp}(\cdot, \cdot) \Big)_{N
  > n,\ \temp \in (0,1]}$ is an LU-precompact family of continuous functions.
\end{lemma}

\begin{lemma}
\label{lem:compactness-of-density-of-mu-kappa}
Let $m\in\Z$. There is an LU-precompact family of continuous functions $\big( h_{n;\temp}(\cdot) \big)_{n>m}$ such that the density of  $\mu_{\temp}
\pi^{-1}_n$ can be expressed as
\begin{equation}
  \label{density-expression-for-mu-kappa}
  \frac{d \, \mu_{\temp} \pi^{-1}_n}{dy}
  =  \frac{Z^{m,n}_{x,y; \temp} e^{- \be h_{n; \temp}(y)}}{ \int_{\R} Z^{m,n}_{x,y'; \temp} e^{- \be h_{n; \temp}(y')} \, dy'}.
\end{equation}
\end{lemma}

\begin{proof}
By Lemma~\ref{lem:approximate-by-finite-support-terminal-measures}, there are terminal measures~$\nu^N_{\temp}$ satisfying~\eqref{eq:finite-support-assumption} such that~$\mu_{\temp}$ is the weak
limit of $\mu^{m,N}_{x, \nu^N_{\temp}; \temp}$.  Suppose $f_{n; \temp}^N(\cdot)$ is the density of $\mu^{m,N}_{x, \nu^N_{\temp}; \temp} \pi^{-1}_n$, 
then by Lemma~\ref{lem:tightness-of-log-density}, $\big( \temp\log f_{n; \temp}^N \big)_{N > n, \temp \in (0,1]}$ is LU-precompact.  Therefore, for each~$\temp$, $\temp\log f_{n; \temp}^N$
converge in LU to some  continuous function $-\tilde{h}_{n; \temp}$ as $N \to \infty$, such that $e^{-\be \tilde{h}_{n; \temp}(y)}$ is the density of $\mu_{\temp}\pi^{-1}_n$.  The family of
functions $\big( \tilde{h}_{n; \temp} \big)_{\temp \in (0,1]}$ is also LU-compact.
One can then define~$h_{n; \temp}(y) = \tilde{h}_{n; \temp}(y) - \temp \ln Z^{m,n}_{x,y;\temp}$ and the lemma follows.
\end{proof}

We are now ready to prove the rest of Theorem~\ref{thm:invisid-limit-for-polymer-measures}.

\begin{proof}[Proof of parts~\ref{item:tightness-of-polymer-measures} and~ \ref{item:convergence-of-polymer-measures} in Theorem~\ref{thm:invisid-limit-for-polymer-measures}]
Part \ref{item:tightness-of-polymer-measures} follows from Lemma~\ref{lem:tightness-IVPM-with-different-temp}. 
Let us prove part~\ref{item:convergence-of-polymer-measures}.  Let $(m,x) \in \ZR$ and $\mu_{\temp} \in \Pc_x^{m,+\infty}(v)$.  Then Lemma~\ref{lem:compactness-of-density-of-mu-kappa} implies that, for each $n >
m$, there is an  LU-precompact family  of continuous functions~$h_{n; \temp}(y)$ such that~\eqref{density-expression-for-mu-kappa} holds.  Suppose $\mu$ is the weak limit of $\mu_{\temp_k}$ for some sequence
$\temp_k \downarrow 0$.  Using a diagonal sequence argument, we see that there is a further subsequence~$\temp'_k\downarrow 0$ such that
for every $n>m$, $h_{n; \temp'_k}(y)$ converge in LU to some~$h_n(y)$ as $\temp'_k\downarrow 0$.

For~$\eps > 0$, let us define the set of paths 
\begin{equation*}
\Lambda_{\eps}^n = \{ \gamma \in S_{x,*}^{m,+\infty}:  A^{m,n}(\gamma)  - A^{m,n} (\gamma_m,\gamma_n)> \eps  \},
\end{equation*}
where $A^{n_1, n_2}(x_1, x_2)$ denotes the minimal action between~$(n_1, x_1)$
and $(n_2, x_2)$.
Then we have 
\begin{equation*}
  \mu_{\temp} (\Lambda^n_{\eps})
  = \frac{ \int    Z^{m,n}_{x,y; \temp} (\Lambda^m_{\eps}) e^{- \be h_{n; \temp}(y)}
    \, dy  }{  \int Z^{m,n}_{x,y; \temp}e^{- \be h_{n; \temp}(y)} \,  dy
  }.
\end{equation*}

For every $\delta > 0$, there exists $L> 0$ such that 
$\mu_{\temp}\big(   B_L^{m,n} \big) \ge 1- \delta$ for all~$\temp \in (0,1]$,
where $B_L^{n,m} = \{ \gamma: |\gamma_i| \le L,\ m \le i \le n \}$.
Also, when~$\temp_k'$ is sufficiently small, we have 
\begin{equation*}
  |  h_{n; \temp'_k}(y) - h_n(y) | \le \eps/4, \quad |y| \le L.
\end{equation*}
Therefore, when $\temp_k'$ is small, 
\begin{equation}
\label{eq:3}
\begin{split}
  \mu_{\temp_k'}(\Lambda^n_{\eps})
  & \le    \mu_{\temp_k'}\big(  (B_L^{m,n} )^c \big)    +  \mu_{\temp_k'}(\Lambda^n_{\eps} \cap B_L^{m,n} )\\
&\le \delta + 
  \frac{ \int_{|y| \le L}    Z^{m,n}_{x,y; \temp_k'} (\Lambda^n_{\eps}  \cap B^{m,n}_L) e^{-(\temp_k')^{-1} h_{n; \temp_k'}(y)}
    \, dy  }{  \int_{|y| \le L} Z^{m,n}_{x,y; \temp_k'}  e^{- (\temp_k')^{-1} h_{n; \temp_k'}}\,  dy
  } \\
 & \le \delta +   e^{ (\temp_k')^{-1} \eps/2}\frac{ \int_{|y| \le L}    Z^{m,n}_{x,y; \temp_k'} (\Lambda^n_{\eps}  \cap
  B^{m,n}_L)  e^{- (\temp_k')^{-1} h_n(y)}
    \, dy  }{  \int_{|y| \le L} Z^{m,n}_{x,y; \temp_k'}  e^{-(\temp_k')^{-1} h_n(y)}\,  dy
  }.
\end{split}
\end{equation}
Due to the continuous dependence of action on paths and compactness of the set $[-L,L]$,
there is~$\eps_1>0$ such that, for each minimizer from $(m,x)$ to $(n,y)$, $|y| \le L$, 
the action of every path in the
$\eps_1$-neighborhood of that minimizer 
is at most $A^{m,n}(x,y) +
\eps/4$.  
(Here, if $\gamma^{*}$ is a path in $S^{m,n}_{*,*}$, its $\eta$-neighborhood is the set~$\{
\gamma \in S_{*,*}^{m,n}: |\gamma_k- \gamma^{*}_k| \le \eta, \ m \le k \le n \}$.)
Therefore, 
\begin{equation}
\label{eq:4}
Z^{m,n}_{x,y; \temp_k'} \ge \eps_1^{n-m} e^{- (\temp_k')^{-1} \big(  A^{m,n}(x,y) + \eps/4 \big)}, \quad |y| \le L.
\end{equation}
On the other hand,  one has
\begin{equation}
\label{eq:5}
 Z^{m,n}_{x,y; \temp_k'} \big(  \Lambda^n_{\eps} \cap B_L^{m,n}  \big)
  \le L^{n-m} e^{- (\temp_k')^{-1} \big(  A^{m,n}(x,y) - \eps \big)}, \quad |y| \le L.
\end{equation}
Combining~(\ref{eq:3}), (\ref{eq:4}) and~(\ref{eq:5}) together, we have 
\begin{equation*}
\mu_{\temp_k'}  (\Lambda^n_{\eps})
  \le \delta + \big(  L/\eps \big)^{n-m} e^{- (\temp_k')^{-1} \eps/4}.
\end{equation*}
Since $\Lambda^n_{\eps}$ is an open set, by weak convergence of $\mu_{\temp_k'}$, we
  have 
\begin{equation*}
\mu( \Lambda^n_{\eps}) \le \liminf_{k \to \infty} \mu_{\temp_k'}(\Lambda^n_{\eps})
\le \delta.
\end{equation*}
Since $\delta$ is arbitrary, we obtain $\mu(\Lambda^n_{\eps})= 0$.

The fact that $\mu(\Lambda^n_{\eps}) = 0$ for every $n$ and~$\eps$ implies that $\mu$ must be a measure on $S^{m,+\infty}_{x,*}$ that 
concentrates on semi-infinite minimizers. To identify the slope, we use Lemma~\ref{lem:uniform-decay-for-marginal-measure} and take~$\temp = \temp_k' \downarrow 0$ in \eqref{eq:uniform-decay-for-marginal-measure} and conclude
that for $\eps > 0$ and $n > n_0(\omega, m, [x], [|v| + \eps], [2\eps^{-1}])$, 
\begin{equation*}
\mu \big(  [(m+n)(v-\eps), (m+n)(v+\eps)]^c \big) = 0.
\end{equation*}
This shows that $\mu$ concentrates on the semi-infinite minimizers in~$\Pc^{m,+\infty}_{x;\temp}(v)$
and completes the proof of part~\ref{item:convergence-of-polymer-measures}.
\end{proof}

\begin{proof}[Proof of Theorem~\ref{thm:convergence-of-busemann-function}]
Part~\ref{item:convergence-of-PM-countable-temperature} follows from Theorem~\ref{thm:invisid-limit-for-polymer-measures}.

For any~$p \in \Z$, 
by~(\ref{eq:compactness-control-by-terminal-measure-all-space-time}) in
Theorem~\ref{thm:all-space-time-point-compactness}, for $(N_2-n) /2 \ge N_1 \ge n_1(n, p) =
n_0(\omega, n, p, [|v|+1],1)$, 
\begin{multline*}
\mu_{y, \nu; \temp}^{n, N_2} \pi_{n+1}^{-1} \big(   [- (|v| + R_1 + 2)N_1, (|v| + R_1 + 2)
N_1]^c\big)  \\
\le \nu \big(  [-(|v|+1)N_2, (|v|+1)N_2]^c \big) + 2e^{-\sqrt{N_1}}, 
\end{multline*}
for every terminal measure $\nu$,  all $\temp \in (0,1]$ and all~$y \in [p,p+1]$.  Taking $\nu = \delta_{N_2 v}$ and letting
$N_2 \to \infty$, we obtain 
\begin{multline}\label{eq:tail-estimate-for-one-step}
\mu_{y; v, \temp}^{n,+\infty}  \pi_{n+1}^{-1} \big(   [- (|v| + R_1 + 2)N_1, (|v| + R_1 + 2)
N_1]^c\big)
\le 2e^{-\sqrt{N_1}}, \\ y \in [p,p+1],\ N_1 \ge n_1(n,p).
\end{multline}
Combining this estimate with~(\ref{eq:derivative-of-U-n-positive-temperature}), we see that~$\big(  u_{v;\temp}(n,\cdot) \big)_{\temp \in (0,1]}$ is uniformly bounded on compact sets.

The first part of the theorem implies that  if 
$(n,y) \not\in \Nc$, then $\mu_{y;v , \temp}^{n,+\infty}$  converges weakly to $\delta_{\gamma_y^{n,+\infty}(v)}$.  Then
combining~(\ref{derivative-of-U-n-zero-temperature}), (\ref{eq:derivative-of-U-n-positive-temperature})
and~(\ref{eq:tail-estimate-for-one-step}), we obtain that 
\[
 u_{v;\temp}(n,y)= \int_\R ( z -y)  \pi_{y; v, \temp}^{n,+\infty}  \pi_{n+1}^{-1}(dz)
  \to  \int_\R ( z -y)  \delta_{\gamma^{n,+\infty}_y(v)} \pi_{n+1}^{-1}(dz) =u_{v;0} (n,y)
\]
for $(n,y) \not \in \Nc$.
Since $\Nc$ is at most countable, $u_{v;\temp}(n,\cdot)$ converges to
$u_{v;0}(n,\cdot)$ at a.e.\ $y$. This implies convergence in~$\GG$ and completes the proof of part~\ref{item:convergen-of-global-solution}.

Finally we will prove part~\ref{item:convergence-of-partition-function-ratio}.
Since the functions $G_{v,\temp}$ and $B_v$ satisfy the relations~(\ref{eq:cocycle-property-for-positive-temp-busemann})
and~(\ref{eq:cocycle-property-for-zero-temp-busemann}), respectively,
it suffices to show the following two limits hold: 
\begin{gather}
  \label{eq:n1-n2-equal}
\lim_{\mathcal{D} \ni\temp \downarrow 0}  -\temp \ln G_{v,\temp} \big(  (n,x), (n,0) \big) =  B_v \big(  (n,x),
(n,0) \big), \quad n \in \Z,\ x\in\R, \\
\label{eq:n1-n2-not-equal}
\lim_{\mathcal{D} \ni \temp \downarrow 0}- \temp \ln G_{v,\temp} \big(  (m,x), (n,0) \big) = B_v \big(  (m,x), (n,0)
\big), \quad n > m, \ x\in\R,
\end{gather}

We recall $U_{v,\temp}$, $\temp \in [0,1]$, defined in
Theorems~\ref{thm:zero-temperature-infinite-minimizers} and~\ref{thm:positive-temperature-IVPM}.
The limit~(\ref{eq:n1-n2-equal}) is equivalent to $U_{v,0}(n,x) = \lim_{\temp \downarrow 0}U_{v,\temp}(n,x)$.

Having shown that~$\big( u_{v;\temp}(n,\cdot) \big)_{\temp \in (0,1]}$ is uniformly bounded and that
$u_{v;\temp}(n,\cdot)$ converge to $u_{v;0}(n,\cdot)$ a.e.\ as $\temp \downarrow 0$, we can use bounded convergence
theorem to conclude that
\begin{equation*}
U_{v;\temp}(n,x) = \int_0^x u_{v;\temp}(n,y) \, dy \to U_{v;0}(n,x) = \int_0^x u_{v;0} (n,y) \, dy, \quad
\temp \downarrow 0.
\end{equation*}
This proves~(\ref{eq:n1-n2-equal}), and the convergence is in LU.

To prove~\eqref{eq:n1-n2-not-equal}, we fix $n > m$ and define $H_{\temp}(x) = - \kappa \ln G_{v;\temp} \big(  (m,x), (n,0)\big)$,  $\temp \in
(0,1]$, and ${H_0(x) = B_v\big(  (m,x), (n,0) \big)}$. 
We are going to show that $\big( H_{\temp}(\cdot) \big)_{\temp \in \mathcal{D}}$ is LU-precompact,
and that $\lim_{\temp \downarrow 0} H_{\temp}(x) = H_0(x)$ for $x \not\in \Nc$ (and hence for a.e.~$x$).
Then the the convergence will hold for all~$x$ and ~(\ref{eq:n1-n2-not-equal}) will follow.

As a consequence of  Lemma~\ref{lem:tightness-of-log-density} applied to $\nu_N = \delta_{vN}$, we
see that the family $ (\temp \ln Z^{n,N}_{y,vN; \temp}/Z^{m,N}_{x,vN; \temp})_{\temp \in \mathcal{D},\ N >  n}$ is
 LU-precompact in $C(\R^2)$ in the variables $x$ and $y$.  Hence, by
part~\ref{item:ration-of-partition-function} of Theorem~\ref{thm:positive-temperature-IVPM} and the
condition~$\omega \in \hat{\Omega}_v \subset \Omega_{v; \temp}$, 
we have that
$\big( \temp\ln G_{v,\temp}\big( (n,y), (m,x) \big) \big)_{\temp \in \mathcal{D}}$ is LU-precompact.
This shows the LU-precompactness of~$\big( H_{\temp} \big)_{\temp \in \mathcal{D}}$.

Using~(\ref{eq:partition-function-ration-solves-burges}), we have 
\begin{multline*}
  G_{v; \temp} \big(  (m,x), (n,0) \big) = \int_{\R} Z^{m,n}_{x,y; \temp} e^{- \be U_{v;\temp}(n,y) } \, dy
 \\ =  \int_{\gamma \in  S_{x,*}^{m,n}} e^{-\be \big(  A^{m,n}(\gamma) +    U_{v;\temp}(n,y) \big)} \, d \gamma.
\end{multline*}

For every $\delta > 0$, there is $L > 0$ such that $\mu_{x; v, \kappa}^{m,+\infty} \big(
B_L^{m,n} \big) \ge 1 - \delta$ for all $\kappa$.  Then
\begin{multline*}
 \int_{\gamma \in  S_{x,*}^{m,n} \cap B_L^{m,n}} e^{-\be \big(  A^{m,n}(\gamma) +
    U_{v;\temp}(n,y) \big)} \, d \gamma \\ \ge (1-\delta)  \int_{\gamma \in  S_{x,*}^{m,n}} e^{-\be \big(  A^{m,n}(\gamma) +
    U_{v;\temp}(n,y) \big)} \, d \gamma,
\end{multline*}
which follows from
\begin{align*}
\mu_{x; v, \kappa}^{m,+\infty} \big(
B_L^{m,n} \big)
&= \frac{ \int_{|y| \le L} \mu^{m,n}_{x,y; \temp} (B_L^{m,n}) Z^{m,n}_{x,y;\temp} e^{-\temp^{-1}
    U_{n,\temp}(y)} dy
}{\int_{\R}
  Z^{m,n}_{x,y';\temp} e^{-\temp^{-1} U_{n,\temp}(y')}dy'}\\
&=\frac{\int_{\gamma \in  S_{x,*}^{m,n} \cap B_L^{m,n}} e^{-\be \big(  A^{m,n}(\gamma) +
    U_{v;\temp}(n,y) \big)} \, d \gamma }{\int_{\gamma \in  S_{x,*}^{m,n}} e^{-\be \big(  A^{m,n}(\gamma) +
    U_{v;\temp}(n,y) \big)} \, d \gamma}.
  \end{align*}
Therefore, 
\begin{equation}
  \label{eq:8}
  \begin{split}
  G_{v; \temp} \big(  (m,x), (n,0) \big)
  &\le (1-\delta)^{-1}
   \int_{\gamma \in  S_{x,*}^{m,n} \cap B_L^{m,n}} e^{-\be \big(  A^{m,n}(\gamma) +
    U_{v;\temp}(n,y) \big)} \, d \gamma \\
  &\le (1-\delta)^{-1}  L^{ m-n }
  e^{ - \be \inf_{|y| \le L} \{ A^{m,n}(x,y) + U_{v;\temp}(n,y) \}}.
  \end{split}  
\end{equation}
By~(\ref{eq:n1-n2-equal}) and~(\ref{eq:busemann-function-solving-invisid-burgers}),
\begin{equation}\label{eq:9}
  \begin{split}
      \liminf_{ \temp \to 0}\inf_{|y| \le L} \{A^{m,n}(x,y) + U_{n,\temp}(y)\}
&  \ge \inf_{|y| \le L} \{ A^{m,n}(x,y) + B_v \big(  (n,y), (n,0) \big) \}\\
&  \ge B_v \big(  (m,x),(n,0) \big).
  \end{split}
\end{equation}
Taking logarithm and multiplying by $-\temp$ in~(\ref{eq:8}) 
and using~(\ref{eq:9}), we obtain that 
\begin{equation*}
\liminf_{\temp \to 0} H_{\temp}(x) \ge H_0(x).
\end{equation*}
Let us fix $\eps > 0$ and define 
\begin{equation*}
  y_0 = \big(  \gamma_x^{m,+\infty}(v) \big)_n
  = \argmin\limits_y \{  A^{m,n}_{x,y} + U_{n,0}(y) \}.
\end{equation*}
There is an
$\eps_1$-neighborhood of  
$\pi_{m,n} \big( \gamma_x^{m,+\infty}(v)  \big)$  such that for each
path~$\gamma$ in this neighborhood, 
\begin{equation*}
|A^{m,n}(\gamma) - A^{m,n}(x,y_0) | \le \eps.
\end{equation*}
Also, by the continuity of~$U_{v;0}(n,\cdot)$ and the LU-convergence of~$U_{v;\temp}(n,\cdot)$ to~$U_{v;0}(n,\cdot)$, there is $\eps_2 >0$ such that  when $\temp$ is small enough we have 
\begin{equation*}
| U_{v;\temp}(n,y) - U_{v; 0}(n,y_0)| \le \eps
\end{equation*}
for every $|y - y_0| \le \eps_{2}$.
Therefore, 
\begin{equation*}
    \begin{split}
G_{v; \temp} \big(  (m,x), (n,0) \big)
  &\ge \big(  \eps_1 \wedge \eps_2 \big)^{n-m}
  e^{- \be \big(  A^{m,n} (  x,y_0 ) + U_{v;0}(n,y_0)  + 2\eps \big)}\\
&  =  \big(  \eps_1 \wedge \eps_2 \big)^{n-m}
  e^{- \be \big(  B_v \big(  (m,x), (n,0) \big) + 2\eps \big)}.    
  \end{split}
\end{equation*}
This implies that 
\begin{equation*}
\limsup_{\mathcal{D} \ni \temp \to 0} H_{\temp}(x) \le H_0(x) + 2\eps.
\end{equation*}
Since $\eps$ is arbitrary, this concludes the proof.
\end{proof}

In the end of this section we give the proof of the technical
Lemma~\ref{lem:tightness-of-log-density}.

\begin{proof}[Proof of Lemma~\ref{lem:tightness-of-log-density}]
We define 
\begin{equation*}
g^N_{\temp}(x,y) = \big( Z^{m,n}_{x,y; \temp} \big)^{-1}f^N_{m,n; \temp}(x,y) =   \int_{-cN}^{cN} \frac{ Z^{m,N}_{y,z; \temp}}{
  Z^{n,N}_{x,z; \temp}}\, \nu_{\temp}^N(dz).
\end{equation*}
It suffices to show that $\Big(  \temp \ln g^N_{ \temp}(\cdot, \cdot) \Big)_{N
  > n , \temp \in (0,1]}$ is LU-precompact.
  
Let us consider a compact set~$K = [p,p+1]\times [-k,k]$.
Denoting~$r = c+R_1+2$,  for $\varepsilon \in (0, 1/2)$, let us define
\begin{equation*}
s_1 = \max\Big\{ n - m,\, n_0 (\omega, n,p,[c +1], 1),\,   \frac{k}{r},\, \ln^2 \frac{\varepsilon}{16}\Big\}
\end{equation*}
and
\begin{equation*}
  s_2 = \max \Big\{ n_0(\omega,n,i,[c + 1], 1):\ |i| \le rs_1 +1\Big\} \vee \ln^2\frac{\varepsilon}{16},
\end{equation*}
where the random function~$n_0$ is introduced in Theorem~\ref{thm:all-space-time-point-compactness}.

We will need several truncated integrals: 
\begin{align*}
&\bar{Z}^{n,N}_{y,z; \temp} = \int_{-rs_2}^{rs_2}
  Z^{n,n+1}_{y,w; \temp}  Z^{n+1,N}_{w,z; \temp}\,  dw  =\int_{-rs_2}^{rs_2} 
    e^{-  \be [\frac{(w-y)^2}{2} +  F_{n+1}(w)]}  Z^{n+1,N}_{w,z,\temp} \, dw,  
  \\
& \bar{Z}^{m,n}_{x,y; \temp} = 
\begin{cases}
 Z^{m,n}_{x,y; \temp},&  n=m+1,  \\
\int_{-rs_1}^{rs_1} Z^{m,m+1}_{x,w; \temp}  Z^{m+1, n}_{w,y; \temp} \,  dw
  ,&  m > n+1,
\end{cases}
\\ 
 & \bar{Z}^{m,N}_{x,z; \temp} = \int_{-rs_1}^{rs_1}
   \bar{Z}^{m,n}_{x,y; \temp}   \bar{Z}^{n,N}_{y,z; \temp} \, dy,\\
  &   \bar{g}^N_{\temp}(x,y) =\int_{-cN}^{cN} \frac{ \bar{Z}^{n,N}_{y,z; \temp}}{
  \bar{Z}^{m,N}_{x,z; \temp}}\, \nu_{\temp}^N(dz).
\end{align*}

For $N > n$, we also define $h_{\varepsilon; \temp}^N = \temp \ln \bar{g}_{\temp}^N$ and~$\tilde{K} = [p,p+1]\times[-rs_1,rs_1] \supset K$.
If we can show that for every $\eps>0$,  all large $N$, and  all $\temp \in (0,1]$,
\begin{equation}\label{eq:eps-approximation-of-g}
    |\temp \ln g^N_{\temp}(x,y)- h_{\varepsilon; \temp}^N(x,y) | \le \varepsilon,\quad (x,y) \in
    \tilde{K}, 
  \end{equation}
and  that $\big( h_{\varepsilon; \temp}^N\big)$ is precompact in $C(\tilde{K})$, then the lemma
will follow  since, given any~${\eps>0}$, we will be able to use an $\eps$-net for~$(h_{\eps; \temp}^N)$ to construct a $2\eps$-net for~$( \temp \ln g_{\temp}^N)$.

  Let $N > \max\{m+2s_1, n+2s_2\}$.
  If $|y| \le rs_1$ and $|z| \le cN$,  then from~(\ref{eq:compactness-control-by-terminal-measure-all-space-time}) with~$v'=0$,
  $u_1 = c+1$, $u_0 = c$, $\nu = \delta_{z}$ and using~$\delta_z([-cN, cN]^c) = 0$,
  we obtain 
\begin{equation*}
  1 - \frac{\bar{Z}_{y,z; \temp}^{n,N}}{Z_{y,z; \temp}^{n,N}} 
  = \mu_{y,z; \temp}^{n, N} \pi_{n+1}^{-1} ([-rs_2, rs_2]^c) \le
    2e^{- \be \sqrt{s_2}} \le \varepsilon/8,  \quad \temp \in (0,1].
  \end{equation*}  
Then,  using the elementary inequality $|\ln(1+x)| \le 2|x|$ for $|x| \le
1/2$ we find
\begin{equation}
  \label{eq:estimate-on-above}
e^{-\varepsilon/4} \le \bar{Z}_{y,z; \temp}^{n,N} / Z_{y,z; \temp}^{n,N}  \le 1.
\end{equation}
Let 
\begin{equation*}
\tilde{Z}^{m,N}_{x,z; \temp} = \int_{-rs_1}^{rs_1} \bar{Z}^{m,n}_{x,y; \temp} Z^{n,N}_{y,z; \temp} dy.
\end{equation*}
Then (\ref{eq:estimate-on-above}) implies
\begin{equation}
  \label{eq:ratio-Z-tilde-and-Z-bar}
1\le \tilde{Z}_{x,z; \temp}^{m,N} / \bar{Z}_{x,z; \temp}^{m,N} \le e^{\varepsilon/4}.
\end{equation}
Similarly, if $x \in [p,p+1]$ and $|z| \le cN$, by~(\ref{eq:compactness-control-by-terminal-measure-all-space-time}),
we obtain
\begin{equation*}
 1- \frac{ \tilde{Z}_{x,z; \temp}^{m,N} }{Z_{x,z; \temp}^{m,N}}
  \le \mu_{x,z; \temp}^{m,N} \pi_{m+1}^{-1} ([-rs_1, rs_1]^c) + \mu_{x,z; \temp}^{m,N} \pi_n^{-1}([-rs_1,rs_1]^c)
 \le 4e^{- \be \sqrt{s_1}} \le \varepsilon/4.
\end{equation*}
Therefore, 
\begin{equation}
  \label{eq:estimate-on-below}
e^{-\varepsilon/2} \le \tilde{Z}_{x,z; \temp}^{m,N}/ Z_{x,z; \temp}^{m,N} \le e^{\varepsilon/2}.
\end{equation}
Combining (\ref{eq:estimate-on-above}), (\ref{eq:ratio-Z-tilde-and-Z-bar}) and (\ref{eq:estimate-on-below}) we obtain 
\begin{equation*}
e^{-\varepsilon} \le \bar{g}^N_{\temp}(x,y) / g^N_{\temp}(x,y) \le e^{\varepsilon},
\end{equation*}
and (\ref{eq:eps-approximation-of-g}) follows.

The next step is to show that $\big( h_{\eps; \temp}^{N} \big)$ is precompact.
For any $|w| \le rs_2$ and $y,y' \in [-rs_1, rs_1]$, we have
\begin{equation*}
\bigg| \frac{(y-w)^2}{2} - \frac{(y'-w)^2}{2} \bigg| \le r(s_1+s_2) |y-y'|.
\end{equation*}
Hence, the definition of $\bar{Z}_{\cdot, z; \temp}^{n,N}$ implies that
\begin{equation*}
 \big|  \temp\ln \bar{Z}^{n,N}_{y,z; \temp} - \temp  \ln\bar{Z}_{y',z; \temp}^{n,N} \big| \le r(s_1+s_2) |y-y'|.
\end{equation*}
Similarly, for all $x,x' \in [p,p+1]$, we have
\begin{equation*}
  \big|   \temp\ln \bar{Z}^{m,N}_{x,z; \temp} - \temp \ln\bar{Z}^{m,N}_{x',z; \temp} \big| \le     (rs_1 + |p| + 1 )|x-x'|. 
\end{equation*}
Combining these two inequalities we see that
\begin{equation}\label{eq:Lip-continuous}
  | h_{\varepsilon; \temp}^N(x,y) - h^N_{\varepsilon; \temp}(x',y') | \le L (|x-x'|+|y-y'|)
\end{equation}
for $L = r(s_1+s_2) + |p| + 1$.
So,  $h_{\varepsilon; \temp}^N$ are uniformly Lipschitz continuous and hence equicontinuous
on $\tilde{K}$.

It remains to show that $h_{\varepsilon; \temp}^{N}$ are uniformly bounded.
Let
\begin{equation*}
 \bar{f}_{\temp}^N(x,y) =    \int_{-cN}^{cN} \frac{  \bar{Z}^{m,n}_{x,y; \temp} \bar{Z}^{n,N}_{y,z; \temp}}{
  \bar{Z}^{m,N}_{x,z; \temp}} \, \nu^N_{\temp}(dz) = \exp \big(   {\textstyle \be h_{\eps; \temp}^N(x,y) } \big)  \bar{Z}_{x,y; \temp}^{m,n}.
\end{equation*}
For each $x \in [p,p+1]$, we have $\int_{-rs_1}^{rs_1} \bar{f}_{\temp}^N(x,y')\,  dy' = 1$.
Let  
\begin{equation*}
M = \sup \{ |\temp\ln \bar{Z}^{m,n}_{x,y ; \temp}|: \temp \in (0,1], (x,y) \in \tilde{K}\}.
\end{equation*}
It is easy to see that~$M < \infty$ a.s. 
Then, by~(\ref{eq:Lip-continuous}) we have for $y,y' \in [-rs_1,rs_1]$, 
\begin{equation*}
  e^{-\be \big(  L\cdot 2r s_1 + M\big)} \bar{f}_{\temp}^N(x,y') \le 
  e^{\be h_{\eps; \temp}^N(x,y)}
  \le \bar{f}_{\temp}^N(x,y') e^{\be \big(  L\cdot 2r s_1 + M\big)}
\end{equation*}
Integrating this inequality over $y'\in [-rs_1, rs_1]$ gives us
\[
   e^{-\be \big(  L\cdot 2r s_1 + M\big)} \le 2rs_1   e^{\be
     h_{\eps; \temp}^N(x,y)}
   \le   e^{\be \big(  L\cdot 2r s_1 + M\big)}.
 \]
 Taking the logarithm gives~$|h_{\eps; \temp}^N(x,y)|  \le L\cdot 2rs_1 + M + |\ln (2rs_1)|$, 
 so~$|h_{\eps; \temp}^N(x,y)|$ are uniformly bounded on $\tilde{K}$.
\end{proof}

\bibliographystyle{alpha}
\bibliography{Burgers}

\end{document}